\documentclass[pdflatex,sn-mathphys-num]{sn-jnl}

\usepackage{graphicx}%
\usepackage{multirow}%
\usepackage{amsmath,amssymb,amsfonts}%
\usepackage{amsthm}%
\usepackage{mathtools}
\usepackage{mathrsfs}%
\usepackage[title]{appendix}%
\usepackage{xcolor}%
\usepackage{textcomp}%
\usepackage{manyfoot}%
\usepackage{booktabs}%
\usepackage{algorithm}%
\usepackage{algorithmicx}%
\usepackage{algpseudocode}%
\usepackage{listings}%
\usepackage{url}
\usepackage{enumitem}
\setlist[enumerate]{leftmargin=.5in}
\setlist[itemize]{leftmargin=.5in}

\theoremstyle{thmstyleone}%
\newtheorem{theorem}{Theorem}
\newtheorem{lemma}[theorem]{Lemma}
\newtheorem{corollary}[theorem]{Corollary}
\newtheorem{proposition}[theorem]{Proposition}%

\theoremstyle{thmstylethree}%
\newtheorem{remark}{Remark}%

\theoremstyle{thmstylethree}%
\newtheorem{definition}{Definition}%

\DeclareMathOperator*{\argmin}{arg\,min}

\DeclarePairedDelimiter{\abs}{\lvert}{\rvert}
\newcommand{\norm}[1]{{\left\lVert{#1}\right\rVert}}
\newcommand{\id}{\mathop{}\mathopen{}\mathrm{id}}
\newcommand{\push}[2]{{#1}_{\#}#2}

\newcommand{\vp}{\mathcal{\varphi}}
\newcommand{\X}{\mathcal{X}}
\newcommand{\Y}{\mathcal{Y}}
\newcommand{\R}{\mathbb{R}} 
\renewcommand{\P}{\mathbb{P}}
\newcommand{\D}{\mathscr{D}}
\newcommand{\B}{\mathscr{B}}
\newcommand{\F}{\mathscr{F}}
\newcommand{\Pc}{\mathbb{P}_{\textup{cost}}}
\newcommand{\Rm}[1]{\textup{\uppercase\expandafter{\romannumeral#1}}}

\begin{document}

\title[The Signed Barycenter]{The Signed Wasserstein Barycenter Problem}

\author*[1]{\fnm{Matt} \sur{Jacobs}}\email{majaco@ucsb.edu}
\equalcont{These authors contributed equally to this work.}

\author[1]{\fnm{Bohan} \sur{Zhou}}\email{bhzhou@ucsb.edu}
\equalcont{These authors contributed equally to this work.}

\affil[1]{\orgdiv{Department of Mathematics}, \orgname{University of California, Santa Barbara}, \orgaddress{\street{South Hall Room 6607}, \city{Santa Barbara}, \postcode{93106}, \state{California}, \country{USA}}}

\abstract{   Barycenter problems encode important geometric information about a metric space.  While these problems are typically studied with positive weight coefficients associated to each distance term, more general signed Wasserstein barycenter problems have recently drawn a great deal of interest.  These mixed sign problems have appeared in a statistical inference setting as a way to generalize least squares regression to measure valued outputs and have appeared in numerical methods to improve the accuracy of Wasserstein gradient flow solvers.  Unfortunately, the presence of negatively weighted distance terms destroys the Euclidean convexity of the unsigned problem, resulting in a much more challenging optimization task.  The main focus of this work is to study properties of the signed barycenter problem for a general transport cost with a focus on establishing uniqueness of solutions.  In particular, when there is only one positive weight, we extend the uniqueness result of  \cite{tornabene2024generalizedwassersteinbarycenters} to any cost satisfying a certain convexity property.   In the case of arbitrary weights, we introduce the dual problem in terms of Kantorovich potentials and provide a sufficient condition for a stationary solution of the dual problem to induce an optimal signed barycenter. }




\maketitle

\section{Introduction}

In Euclidean geometry, affine combinations provide a natural mechanism to interpolate and extrapolate between points. The collection of all such combinations,  the affine span, is rooted in several fundamental ideas in linear algebra such as the basis and dimension. Given the usefulness of these operations, it is natural to ask whether one can extend such notions to more general settings lacking a vector space structure.
In this work, we consider this question in the space of probability measures and approach it using the notion of signed barycenters. Given a collection of probability measures $\{\mu_{i}\}_{i=1}^m$ supported on an open subset $\X\subseteq \R^d$, and a collection of weights $\{a_{i}\}_{i=1}^m\subset \R$ such that $\sum_{i=1}^ m a_{i}=1 $, we define the signed barycenter as a probability measure that minimizes the weighted transport cost:
    \[\bar{\nu}\in \argmin_{\nu} \sum_{i=1}^m a_{i} K_c(\mu_{i},\nu),\]
where $K_c(\mu_{i},\nu)$ represents the cost of transporting input probability measure $\mu_{i}$ to another probability measure $\nu$, subject to a ground cost function $c:\X\times \X\to [0,\infty)$.

When all weights $\{a_{i}\}_{i=1}^m$ are nonnegative, one encounters a traditional barycenter problem in the Wasserstein space, which was first introduced in \cite{Agueh2011Barycenters} by Agueh and Carlier.  Unlike a traditional linear weighted average of probability measures, the Wasserstein barycenter captures much more geometric information about the underlying measures, inviting numerous applications including matching \citep{Carlier2010Matching}, texture mixing \citep{Rabin2012Slice}, shape interpolation \citep{Solomon2015Convolutional}, image processing \citep{feng2024improving, Han2024Learning}, template registration \citep{Boissard2015Template}, clustering \citep{Ye2017Cluster, Zhuang2022Kmeans}, Bayesian inference \citep{Srivastava2018Bayes}, model fusion \citep{Singh2020Model}, synthesis and analysis \cite{Werenski2022Coding}, adversarial classification \citep{GT2023Multimarginal}, among others. 
Allowing negative weights leads to a natural extension under the name \textit{generalized Wasserstein barycenter} studied by \citet{tornabene2024generalizedwassersteinbarycenters}. However, this terminology has been used inconsistently across the literature \citep[see, for example,][]{Alteschuler2022NP, GT2023Multimarginal, Garcia2024Existence, Delon2023GWB}. To avoid ambiguity, we prefer to use the more descriptive term \textit{signed barycenter}.

Despite its broad applicability, the classical Wasserstein barycenter is defined through nonnegative weights and is therefore inherently limited to interpolation in the space of probability measures. This restriction precludes extrapolation and the  corresponding notion of affine span.  Indeed, forming weighted combinations with negative coefficients occur naturally in many applications, including methods as simple and fundamental as least squares regression.  Applying such operations to the space of probability measures has recently attracted attention from various researchers.  For instance,  \citep{Gunsilius2023Distributional} proposes a method to construct a counterfactual quantile function to measure the effects of policy proposals on economic indicators. In this setting, negative coefficients can indicate that a desired variable is negatively correlated with a particular control state. The conditional Wasserstein barycenter problems, studied by \citep{Fan2025Conditional, Ubribe2024Network} where the weights depend on scalar or vector-valued predictors and may take negative values, can be viewed as a special case of the signed barycenter problem in this paper.

One may also regard the signed barycenter as a map from input measures and weights
    \[(\mu_{1},\ldots,\mu_{m};a_1,\ldots,a_{m})\mapsto \bar{\nu}=\Gamma_{
    \textbf{a}}(\mu_1,\ldots,\mu_{m})\]
into a response measure. This can be interpreted as a form of parametric statistical inference in the space of probability measures, where both the predictor $(\mu_1,\ldots,\mu_{m})$ and the response $\bar{\nu}$ are distributions. Statistical inference models in which both the sample and parameter spaces are functional spaces rather than Euclidean spaces have gained increasing attention with the rise of functional data analysis. Distribution-to-distribution regression, i.e., inferring the distributional response from a distributional predictor, has been an important technique since the introduction of the Wasserstein metric. See \cite{Petersen2022Modeling} for a comprehensive review and \cite{kim2025dfnndeepfrechetneural} for a complete summary of regression models under different settings.

Signed barycenter problems are also closely related to a key component of high-order schemes for the Wasserstein gradient flow. The renowned JKO scheme \citep{JKO1998FPequation}, $\mu_{m+1}\in \argmin_{\mu} \F(\mu)+\frac{1}{2\eta}W_2^2(\mu_{m},\mu)$, can be interpreted as a backward Euler scheme with a stepsize $\eta$ for the Wasserstein gradient flow of a functional $\F$. \citet{Han2023High} extended this first-order scheme to multi-step, multi-stage and mixed schemes, for high-order accuracy and energy stability. Their building block takes the form
    \[\mu_{n+1} \in \argmin_{\mu} \F(\mu)+\frac{1}{2\eta}\sum_{j=1}^K a_{j}W_2^2(\nu_{j},\mu),\]
where $\{a_j\}_{j=1}^K$ are prescribed, possibly negative, while measures $\{\nu_j\}_{j=1}^K$ ultimately depend only on measures $\{\mu_{i}\}_{i=1}^n$ in the previous iterates. While their work primarily focuses on stability conditions, the existence and uniqueness of this variational problem remain open. Several second-order schemes had previously been explored, including implicit midpoint \citep{Legendre2017Second},  Crank–Nicolson \citep{Carrillo2022PD}, BDF2 \citep{Matthes2019BDF2}, EVBDF2 \citep{Gallouet2024EVBDF}. Notably, the extrapolation components in BDF2 and EVBDF2 correspond to a special case of the signed barycenter problem, with exactly one positive and one negative weight.

\subsection{Main results}

In this paper, we study (empirical) signed barycenter problems using both a primal approach in the space of probability measures and a dual approach in the space of potential functions. 

\vspace{0.5em}
\noindent\textbf{Main Assumptions.} 
Throughout this work, we assume that $\X, \Y \subset\R^d$ are open sets and the cost function $c:\X\times\mathcal{Y}\to [0,\infty)$ satisfies the following properties.
\begin{enumerate}
    \item \label{asp:cost_twist} $c\in C^2(\X\times\Y)$ and is bi-twisted, i.e., $y\mapsto \nabla_x c(x,y)$ and $x\mapsto \nabla_y c(x,y)$ are injective.
    
    \item \label{asp:cost_compact} The sets $\left\{x\in \X: c(x,y')\leqslant b\right\}$ and $\left\{y\in \mathcal{Y}: c(x',y)\leqslant b\right\}$ are compact for any $b\in \R$ and $(x',y')\in \X\times\Y$.

    \item \label{asp:cost_growth} There exists a point $(x_0,y_0)\in \X\times\Y$ such that for any $\varepsilon>0$ there exists $B_{\varepsilon}>0$ such that
    \begin{subequations}
    \begin{align}
    (1-\varepsilon)c(x_0,y)-B_{\varepsilon} c(x,y_0)\leqslant &c(x,y)\leqslant (1+\varepsilon)c(x_0,y)+B_{\varepsilon} c(x,y_0) \label{eq:cost_epsilon_property}\\
    (1-\varepsilon)c(x,y_0)-B_{\varepsilon} c(x_0,y)\leqslant &c(x,y)\leqslant (1+\varepsilon)c(x,y_0)+B_{\varepsilon} c(x_0,y)\label{eq:cost_epsilon_property_2}
    \end{align}
    \end{subequations}
for all $(x,y)\in \X\times\Y$.

\end{enumerate}
In particular, we may pick $c(x,y)=\abs{x-y}^p$ when $\X=\Y=(\R^d,\norm{\cdot}_p)$.

For $(x_0,y_0)\in \X\times\Y$ identified in Assumption \ref{asp:cost_growth}, it is natural to define the spaces
    \begin{subequations}
    \begin{align}
    \Pc(\X)&:=\left\{\mu\in \P(\X): \int_{\X} c(x,y_0)\, \mathrm{d}\mu(x)<\infty\right\},\label{eq:PcostX}\\
    \Pc(\Y)&:=\left\{\nu\in \P(\Y): \int_{\Y} c(x_0,y)\, \mathrm{d}\nu(y)<\infty\right\},\label{eq:PcostY}\\
    \Pc(\X\times\Y)&:=\left\{\pi\in \P(\X\times\Y): \int_{\X\times\Y} \big(c(x,y_0)+c(x_0,y)\big)\, \mathrm{d}\pi(x,y)<\infty\right\}.\label{eq:PcostXY}
    \end{align}
    \end{subequations}
Readers may recall that these spaces impose finiteness requirements, analogous to those in \citet[Theorem 3.2]{Ambrosio2003Existence}.

\begin{remark}
These definitions and assumptions generalize the signed barycenter problem in the space $(\P_2(\X), W_2(\cdot,\cdot))$ of probability measures with finite moments under the quadratic cost $c(x,y)=\frac{1}{2}\abs{x-y}^2$, to settings of a general cost function. 
\begin{itemize}
    \item Under Assumption \ref{asp:cost_compact}, it is sufficient to check a simple integral condition  for the tightness of a set $K\subseteq \Pc(\X)$. Specifically, $K$ is tight if $\sup_{\mu\in K}\int_{\X} c(x,y_0)\,\mathrm{d}\mu(x)<\infty$ for some $y_0$.
    \item Assumption \ref{asp:cost_growth} is a growth condition. It mirrors formula (6.10)  in \cite{Villani2009Old}, a property of the cost function given by a metric.
\end{itemize}
\end{remark}

\vspace{0.5em}
\noindent\textbf{Primal approach}\indent Under our main assumptions on the cost function $c$ and sets $\X,\Y$, we establish in Theorem \ref{thm:existence} the existence of the signed barycenter, which generalizes the existence result of \cite{tornabene2024generalizedwassersteinbarycenters} for the quadratic cost only. We then analyze properties of the signed barycenter functional on the space of probability measures. The presence of negative weights fundamentally changes the nature of the problem: the functional is no longer Euclidean convex, and the signed barycenter need be neither absolutely continuous nor unique,  as observed in \cite{tornabene2024generalizedwassersteinbarycenters}. Furthermore, for general costs, we lose the tool of $\lambda$-convexity along generalized geodesics that plays a crucial role in previous literature for the quadratic cost (see \cite{Gallouet2024EVBDF, Gallouet2025Metric} for the case of one positive and one negative weight, and \cite{tornabene2024generalizedwassersteinbarycenters} for the case of one positive and multiple negative weights). 

These obstacles motivate a dual approach. Nevertheless, a notable exception arises when the barycenter functional involves exactly one positive weight; see also \cite{tornabene2024generalizedwassersteinbarycenters} for the quadratic cost. We prove in Theorem \ref{thm:convex-unique} that, if there is a convex interpolation for the weighted cost function, then there exists an explicit continuous interpolating curve in the space of probability measures, along which the barycenter functional is convex. In this sense, this result lifts the convex interpolation of the weighted cost to a convex interpolation curve in the spirit of generalized geodesics.

\vspace{0.5em}
\noindent\textbf{Dual approach}\indent We study a dual formulation that is different from the constrained dual formulation in \cite{Agueh2011Barycenters} and the multimarginal dual formulation in \cite{Gangbo1998Multidimensional, Agueh2011Barycenters} for the classical barycenter problem. Our formulation also differs from the Toland-duality used in \cite{Gallouet2024EVBDF,Gallouet2025Metric} for the Wasserstein extrapolation with one positive and one negative weights, where Brenier's theorem is modified by imposing a strong convexity constraint on the Kantorovich potentials. 

Interestingly, in the presence of negative weights, the structure of the dual problem is a min-max optimization problem instead of the usual dual maximization structure.  Hence, our study of the dual problem focuses on properties of stationary points and saddle points of the functional rather than global maximizers. Although this functional is not convex-concave in general, saddle points retain the expected optimality properties: Theorem \ref{thm:main-induced} shows that any saddle point to the dual problem induces an optimal signed barycenter to the primal problem.

The uniqueness of the signed barycenter obtained from the primal approach is rather restrictive: only one positive weight is allowed to utilize the strong convexity along some interpolation curve. Furthermore, regardless of regularity of the input measures, a signed barycenter need not be absolutely continuous. These two limitations make it difficult to obtain theoretical guidance for computing signed barycenters.

Our main results address these difficulties through the dual approach.

\begin{enumerate}
    \item Theorem \ref{thm:unique} shows if there are at least two input measures associated with positive weights are absolutely continuous, and there is a $c$-concave saddle point of the dual problem, then the optimal signed barycenter is unique. Furthermore, the signed barycenter can be characterized by the saddle point explicitly \eqref{eq:bary-po-f}.
\end{enumerate}
From a computational viewpoint, stationary points are more accessible than a saddle point in a minimax problem. We therefore provide a sufficient condition under which a stationary point can be upgraded to a saddle point.

\begin{enumerate}[start=2]
    \item Suppose there are at least two input measures associated with positive weights are absolutely continuous. Theorem \ref{thm:saddle_pt_sufficient} -- Corollary \ref{cor:sufficient} show that under a convexity criterion on the Kantorovich potentials associated with \textit{positive} weights in the $c$-concave stationary point, the stationary point is a saddle point.  Consequently, the signed barycenter problem has a unique solution and furthermore the solution is absolutely continuous.
\end{enumerate}
In the simple setting of 2-Wasserstein metric, this criterion reduces to checking whether the weighted sum of the Kantorovich potentials associated with \textit{positive} weights is $L$-strongly convex, with $L$ explicitly determined by the given weights (see \eqref{eq:additional-po}). 

A central technical challenge in our analysis is still rooted in the fact that a signed barycenter may not be absolutely continuous. As a result, the transport from the barycenter back to the input marginals may not be given by a map. Furthermore, it becomes difficult to deduce certain properties of the Kantorovich potentials on the support of the barycenter, due to the lack of open sets and perturbation directions.

Together with the preceding optimality, uniqueness and regularity guarantees, this provides a theoretical foundation for future numerical methods for signed barycenters. In particular, when the input measures are absolutely continuous, the characterization in \eqref{eq:stationary-by-pf-po} suggests an adaptation the back-and-forth method for optimal transport proposed in \cite{Jacobs2020BF}, which is fundamentally an $H^1$ gradient-ascent scheme for the dual variables.  Note that the dual approach is particularly promising as the number of dual variables grows linearly with the number of input measures, in contrast to multimarginal dual formulations \cite{Zhou2022Efficient,Gangbo1998Multidimensional}, whose pairwise structure leads to quadratic growth in the number of input measures. Indeed, this feature has already proved useful for computing classical Wasserstein barycenters \citep{kim2025sobolev} and is likely to extend to signed barycenters as well.

The rest of the paper is organized in the following manner. In Section \ref{sec:pre}, we introduce the necessary notations and recall preliminary results from optimal transport, with the focus of Kantorovich duality and properties of $c$-transforms. In Section \ref{sec:primal}, we establish the existence of signed barycenters under general cost functions from the primal perspective and analyze the properties of the barycenter functional. In particular, we obtain uniqueness results in the special case of only one positive weight under additional condition on the convex interpolation for general cost functions. In Section \ref{sec:dual}, we develop a dual formulation in terms of Kantorovich potentials. We characterize stationary and saddle points of the dual functional, show that saddle points induce signed barycenters, and provide sufficient conditions for uniqueness and for upgrading stationary points to saddle points.

\section{Preliminaries}\label{sec:pre}

\subsection{Notations}

For notational convenience, we assume that the affine combinations of weights $\left\{a_i\right\}_{i=1}^m$ are nonzero and ordered without loss of generality, i.e., $a_1\geqslant \cdots \geqslant a_\ell>0> a_{\ell+1}\geqslant\ldots \geqslant a_m$, where the index $\ell$ marks the smallest positive weight. $I_{+}=\left\{1,\ldots,\ell\right\}$ and $I_{-}=\left\{\ell+1,\ldots,m\right\}$ are the index sets for positive and negative weights respectively.

We denote by $\mathbb{P}(\X)$ the set of probability measures supported on the set $\X$, and by $\mathbb{P}^{\mathrm{AC}}(\X)$ if they are absolutely continuous with respect to the Lebesgue measure. We sometimes append the subscript like $\Pc(\X)$ in \eqref{eq:PcostX} to indicate that it is a subset of $\mathbb{P}(\X)$ related to the cost function, and we may abuse a measure $\mu$ as its density function $\mu(x)$. We denote by $\Pi(\mu,\nu)$ the set of couplings with the first marginal as $\mu$ and the second marginal as $\nu$. More generally, $\Pi(\mu_1,\ldots,\mu_m, \nu)$ denotes the set of couplings with multiple marginals. Let $\X_1=\cdots=\X_m=\X$ and $P_{\{i\}}:\X_1\times\cdots\times \X_m\times \Y\to \X_i\times \Y$ be a projection. We define $\Pi_*(\mu_1,\ldots,\mu_m; \nu)$ as a collection of couplings, such that for any element $\pi\in \Pi_*(\mu_1,\ldots,\mu_m; \nu)$, it satisfies that for every $i$, $\push{P_{\{i\}}}{\pi}\in \Pi(\mu_i,\nu)$ is an optimal transport plan associated with the Kantorovich problem \eqref{KP} (see Definition \ref{def:opt-coup} for details). We reserve $P$ and its variants for projection maps, defining each explicitly when it is introduced.

We say $\mu_n$ weakly converges to $\mu$ in $\Pc(\X)$, denoted by $\mu_n \rightharpoonup \mu$, if $\mu_n$ weakly converges to $\mu$ in $\P(\X)$ (in duality with bounded continuous test functions in $C_b(\X)$). 

\subsection{Optimal Transport in \texorpdfstring{$\Pc(\X)$}{}}\label{sub:space}

The optimal transport (OT) problem addresses the allocation of resources in a way that minimizes an associated cost function $c$. In this section, we follow \citet[Section 5.1]{Ambrosio2008Gradient} and \citet[Chapter 10]{Villani2009Old} to provide the corresponding classical OT theory under our assumptions on the sets $\X,\Y$ and the cost $c$, where they are possibly unbounded. First, we state some results on measure theory under our assumptions.

By \citet[Theorem 5.1.3 and Remark 5.1.5]{Ambrosio2008Gradient}, to check a set $K\subseteq \Pc(\X)$ is tight or $K\subseteq\Pc(\X)$ is precompact in duality with bounded continuous test functions, it is equivalent to check $\sup_{\mu\in K}\int_{\X}c(x,y_0)\mathrm{d}\mu(x)<\infty$ for some $y_0$. 

Suppose $c(x,y_0)$ is lower semicontinuous and bounded from below, let $\mu_n \rightharpoonup \mu$ in $\Pc(\X)$, and $\sup_n \int_{\X} c(x,y_0)\,\mathrm{d}\mu_n<+\infty$. As a consequence of the Portmanteau theorem \citep[Theorem 1.3.4]{Wellner2023Weak}, $\int_{\X}c(x,y_0)\mathrm{d}\mu \leqslant \liminf \int_{\X} c(x,y_0)\mathrm{d}\mu_n <+\infty$, which yields that $\mu\in \Pc(\X)$. However, it may not be true that $\lim_{n\to\infty} \int_{\X}c(x,y_0)\mathrm{d}\mu_n=\int_{\X} c(x,y_0)\mathrm{d}\mu$. To ensure this, one needs to further check the uniform integrability $\lim_{b\to\infty} \int_{\X\backslash \{x: c(x,y_0)\leqslant b\}}c(x,y_0)\mathrm{d}\mu(x)=0$ uniformly w.r.t $\{\mu_n\}_{n\in \mathbb{N}}$.

\subsubsection{Primal problem and dual problem}

Let $\mu\in \Pc(\X), \nu\in \Pc(\Y)$, the Monge formulation of the optimal transport problem takes the form
    \begin{equation}
        \label{MP}
        M_c(\mu,\nu)=\inf_{T}\int c(x,T(x))\, \mathrm{d}\mu(x),
    \end{equation}
among all measurable maps $T:\X\to\Y$ which transport $\mu$ to $\nu$, i.e., $\push{T}{\mu}=\nu$. The minimizer is called as \textit{Monge map}. Recall the push-forward measure $\push{T}{\mu}$ is defined via 
    \[
    \int_{\Y}f(y)\, \mathrm{d}\left(\push{T}{\mu}\right)(y)=\int_{\X} f(T(x))\, \mathrm{d}\mu(x)
    \]
that holds for any bounded measurable function $f:\Y\to\R$.

Kantorovich relaxed the Monge problem into a convex optimization problem:
    \begin{equation}
        \label{KP}
        K_c(\mu,\nu)=\inf_{\pi \in \Pi(\mu,\nu)}\left\{\int_{\X\times \Y}c(x,y)\, \mathrm{d} \pi(x,y)\right\},
    \end{equation}
where the infimum is taken over the set $\Pi(\mu,\nu)$ of joint distributions. Note that any admissible map in the Monge problem induces a joint distribution $\push{(\id,T)}{\mu}$, where $\id:\X\to\X$ is the identity map. For the quadratic cost $c(x,y)=\abs{x-y}^2$ on $\X=\Y$, the optimal objective value $K_c(\mu,\nu)$ is also denoted by $W_2^2(\mu,\nu)$ and any minimizer $\bar{\pi}$ is called as \textit{optimal transport plan}. The space $(\P(\X),W_2)$ of probability measures is referred to as the \textit{2-Wasserstein space}.

Given $\mu\in \Pc(\X), \nu\in\Pc(\Y)$, $\Pi(\mu,\nu)\subseteq \Pc(\X\times \Y)$, and it holds under Assumption \ref{asp:cost_growth} that for any $\pi\in \Pi(\mu,\nu)$, $K_c(\mu,\nu)\leqslant \int_{\X\times \Y} c(x,y)\mathrm{d}\pi(x,y)\leqslant \int_{\Y}(1+\varepsilon)c(x_0,y)\mathrm{d}\nu +\int_{\X}B_{\varepsilon} c(x,y_0)\mathrm{d}\mu<\infty$. Thanks to Lemma 4.4 in \cite{Villani2009Old} and Corollary 2.9 in \cite{Ambrosio2021Lectures}, $\Pi(\mu,\nu)$ is compact with respect to the weak topology in $\P(\X\times \Y)$. In addition, the stability results \citep[Theorem 5.20]{Villani2009Old} apply here. More precisely, for any $\mu_n\rightharpoonup \mu \in \Pc(\X)$ and $\nu_n\rightharpoonup \nu \in \Pc(\Y)$, one has $K_c(\mu_n,\nu_n)\to K_c(\mu,\nu)$ and the sequence of optimal plans $\pi_n\in \Pi(\mu_n,\nu_n)$ weakly converges to optimal plan $\pi\in \Pi(\mu,\nu)$ in $\Pc(\X\times \Y)$.

The Kantorovich problem admits a dual formulation:
    \begin{equation}
        \label{DP}
        \sup\left\{\int_{\X} f_1\, \mathrm{d}\mu +\int_{\Y} f_2\, \mathrm{d}\nu\,:\, f_1\in C_b(\X), f_2\in C_b(\Y), f_1(x)+f_2(y)\leqslant c(x,y)\right\}.
    \end{equation}

\subsubsection{\texorpdfstring{$c$}{}-transform}\label{sub:OT}

The $c$-transform, arising the study on the dual problem, plays a central role in connecting the Monge problem \eqref{MP}, the Kantorovich problem \eqref{KP} and the dual problem \eqref{DP}. For a comprehensive treatment of $c$-transform and $c$-duality, we refer readers to \cite{Villani2003Topics, Villani2009Old, Santambrogio2015AOT, Ambrosio2021Lectures}. Here, we focus on aspects directly pertinent to our research.

Given $f:\Y\to\R$, the associated $c$-transform $f^c:\X\to [-\infty,\infty)$ of $f$ is defined as
    \begin{equation*}\label{eq:c-tran}
    f^c(x)=\inf_{y\in\Y} c(x,y)-f(y).
    \end{equation*}
Given $g:\X\to \R$, we define $g^c:\Y\to [-\infty,\infty)$ analogously. $f$ is said to be $c$-concave if there exists a function $g$ such that $f=g^c$. When $c(x,y)=\frac{1}{2}\abs{x-y}^2$ in $\R^d$, $f$ is $c$-concave if and only if $y\mapsto \frac{1}{2}\abs{y}^2-f(y)$ is convex and lower semi-continuous. 

\begin{lemma}\label{lem:c-tran-prop}
For $a>0$ and any continuous function $c(x,y)$, the following statements hold
\begin{enumerate}[label=\arabic*)]
\item $f^{cc}\geqslant f$, with equality if and only if $f$ is $c$-concave;
\item Scaling laws: $\displaystyle f^{ac}(x):= \inf_{y\in\Y}ac(x,y)-f(y)=a(\frac{f}{a})^c(x)$;
\item The map $f\mapsto f^c$ is concave, i.e., $[(1-t)f_1+tf_2]^c(x)\geqslant (1-t)f_1^c(x)+tf_2^c(x)$;
\item The map $f\mapsto f^c$ is ordering reverse, i.e., $f\geqslant g$ implies $f^c\leqslant g^c$.
\item If $f$ is bounded from above by a $c$-affine function $c(x_0,\cdot)+a$, then $f^c(x_0)>-\infty$. Consequently, if $f: \Y\to [-\infty,\infty)$ is $c$-concave, then both $f$ and $f^c$ are real-valued.
\item If $c$ satisfies the main assumptions,  for any $f\in C_b(\Y)$, $y\mapsto c(x,y)-f(y)$ is coercive. Consequently, there exists a minimizer $y_0\in \Y$ such that $f^c(x)=c(x,y_0)-f(y_0)$.
\end{enumerate}
\end{lemma}

\begin{proof}
Part 1, see \citet[Proposition 1.34]{Santambrogio2015AOT}.
    
Part 2, $f^{ac}(x)=\inf_{y\in\Y} ac(x,y)-f(y)=a\inf_{y\in\Y} c(x,y)-\frac{f(y)}{a}=a (\frac{f}{a})^c(x).$

Part 3, 
    \begin{equation*}
    \begin{aligned}
    [(1-t)f_1+tf_2]^c(x)&= \inf_{y\in\Y} c(x,y)- (1-t)f_1(y) - tf_2(y)\\
       &=\inf_{y\in\Y} (1-t)[c(x,y)-f_1(y)] + t[c(x,y)-f_2(y)]\\
       &\geqslant (1-t) f_1^c(x) + t f_2^c(x).\\
    \end{aligned}
    \end{equation*}
    
Part 4, $f^c(x)=\inf_{y\in\Y}c(x,y)-f(y)\leqslant \inf_{y\in\Y}c(x,y)-g(y)=g^c(x)$.

Part 5, if for any $y\in \Y$, $f(y)\leqslant c(x_0,y)+a$ for some $x_0,a$, then $c(x_0,y)-f(y)\geqslant c(x_0,y)-c(x_0,y)-a= -a$ implies that $f^c(x_0)=\inf_{y\in \Y} c(x_0,y)-f(y)\geqslant -a >-\infty$. 

Recall that a function is $c$-concave if and only if it is the infimum of a family of $c$-affine functions, i.e., $f(y)=\inf_{x,a} c(x,y)+a\leqslant c(x,y)+a$ for any $x$ and $a$. Consequently, $f^{c}(x)>-\infty$ for any $x\in \X$. As $f^c$ is a $c$-concave function defined on $\X$, one may repeat this argument to obtain $f(y)=f^{cc}(y)>-\infty$.

Part 6, assume $\abs{f(x)}\leqslant M$ and note that $\{y\in \Y: c(x,y)-f(y)\leqslant a\}\subseteq \{y\in \Y: c(x,y)\leqslant M+a\}$. By the main assumption \ref{asp:cost_compact} and the fact that the closed subset of a compact set is compact, we have shown that $\{y\in \Y: c(x,y)-f(y)\leqslant a\}$ is compact for any $a\in \R$. Thus $c(x,y)-f(y)$ is coercive in $y$ and there is a minimizer for $f^c(x)=\inf c(x,y)-f(y)$.
\end{proof}

A direct consequence is that we can rephrase \eqref{DP} as a dual problem over $c$-concave and bounded continuous functions:
    \begin{equation}\label{DP2}
    D_c(\mu,\nu)=\sup\int_{\X} f^c(x)\, \mathrm{d}\mu(x)+\int_{\Y}f(y)\, \mathrm{d}\nu(y).
    \end{equation}
The optimal $c$-concave function $f$ is referred to as the (optimal $c$-concave) Kantorovich potential for the transport of $\nu$ to $\mu$.

Following part 6) in Lemma \ref{lem:c-tran-prop}, and noting that $f^c$ shares the same modulus of continuity as $c$ (see Section 1.2 in \cite{Santambrogio2015AOT}), we can see the $c$-transform induces a map.

\begin{lemma}[\cite{Santambrogio2015AOT, Gangbo2004introduction}]\label{lem:Gangbo}
For any Lipschitz continuous function $c(x,y)$ and any $f(y)\in C_b(\Y)$, if $y\mapsto \nabla_x c(x,y)$ is injective, then
\begin{enumerate}[label=\arabic*)]
    \item Fixed $y_0\in\Y$, $\inf_{x\in\X}c(x,y_0)-f^c(x)$ has a minimizer $x_0$ such that  
    \[\nabla_x c(x_0,y_0)=\nabla f^c(x_0),\]
    provided that $f^c(x)$ is differentiable at $x_0$. Moreover, it induces a map
    \begin{equation}\label{eq:map}
        y_0=T_{f^c,c}(x_0)=(\nabla_x c(x_0,\cdot))^{-1}(\nabla f^c(x_0)).
    \end{equation}
    \item Given any continuous function $\phi:\Y\to\R$, we have
    \[\lim_{\varepsilon\to 0}\dfrac{(f+\varepsilon \phi)^c(x)-f^c(x)}{\varepsilon}=-\phi(T_{f^c,c}(x)).\]
\end{enumerate}
\end{lemma}

Given any differentiable function $f:\Y\to\R$, the injectivity from the bi-twist condition ensures two maps. As we consider a uniform cost $c$ among any term $K_c(\mu_i,\nu)$ in this paper, we omit the second $c$ in the subscript of \eqref{eq:map} and use $T_{f}:\Y\to\X$ and $ T_{f^c}:\X\to \Y$:
\[T_{f}(y_0)=(\nabla_y c(\cdot,y_0))^{-1}(\nabla f(y_0)),\qquad\textrm{and}\qquad T_{f^c}(x_0)= (\nabla_x c(x_0,\cdot))^{-1}(\nabla f^c(x_0)),\]
though two inverse functions $(\nabla_y c(\cdot, y_0))^{-1}, (\nabla_x c(x_0,\cdot))^{-1}$ are used.

We now outline how $c$-transform connects the dual and primal problems. Our presentation follows from \citep{Villani2009Old}, and the main assumptions made here fall within the general assumptions used in Theorem 10.38 there for unbounded cases. The development builds on the celebrated Knott-Smith optimality condition \citep{Knott1984Optimality} and Brenier's theorem \citep{Brenier1991Polar}, and on their subsequent extensions to more general Riemannian settings \citep{Mccann2001Polar, Gangbo1996Geometry}.

\begin{theorem}[\cite{Brenier1991Polar, Mccann2001Polar}]\label{thm:Brenier}
Under our main assumptions on the cost $c$ and sets $\X,\Y$, there exists a locally Lipschitz and $c$-concave function $f:\Y\to\R$ such that:
    \begin{enumerate}[label=\arabic*)]
    \item $f$ is optimal to the dual problem \eqref{DP2}.
    \item The strong duality holds $K_c(\mu,\nu)=D_c(\mu,\nu)$. If $\pi$ is optimal to the Kantorovich problem \eqref{KP}, then it holds almost everywhere that
    \begin{equation}\label{eq:ot-optimality}
        \pi(x,y)(f^c(x)+f(y)-c(x,y))=0.
    \end{equation}
    \item Let $\mu\in \Pc^{\textrm{AC}}(\X)$, then the map $y=T_{f^c}(x)$ defined in \eqref{eq:map} is the unique optimal transport map transporting from $\mu$ to $\nu$, i.e., the Monge map to \eqref{MP}. Furthermore, it induces $\pi=\push{(\id,T_{f^c})}{\mu}$, which is optimal to the Kantorovich problem \eqref{KP}. Consequently, $M_c(\mu,\nu)=K_c(\mu,\nu)=D_c(\mu,\nu)$. 
    \end{enumerate}
\end{theorem}

Lemma \ref{lem:Gangbo} and Theorem \ref{thm:Brenier} shed light on using the dual formulation and its solution to solve the primal problem. \cite{Jacobs2020BF} introduced a $\dot{H}^1$ gradient ascent method in a back-and-forth fashion to solve the dual formulation. For notational convenience, we consider two functionals $J_{\nu}(f)=\int_{\Y} f\, \mathrm{d}\nu$ and $J^c_{\mu}(f)=\int_{\X}f^{c}\, \mathrm{d}\mu$ associated with the dual problem. The first variation and concavity of the dual functional were originally established in \cite{Gangbo1994Elementary, Gangbo1996Geometry}. Here we adopt the statements presented in Theorem 1 and Lemma 3 in \cite{Jacobs2020BF}.

\begin{lemma}[\cite{Gangbo1994Elementary, Gangbo1996Geometry, Jacobs2020BF}]\label{lem:OT-dual-optimality}
Under the main assumptions on the cost $c$ and the sets $\X,\Y$, let $\mu\in \Pc^{\textup{AC}}(\X)$, for any continuous function $f$, 
\begin{enumerate}[label=\arabic*)]
    \item The map $f\mapsto J_{\nu}(f)$ is linear, while the map $f\mapsto J_{\mu}^c(f)$ is concave.
    \item The first variations of $J_{\nu}(f)$ and $J_{\mu}^c(f)$ are $\delta J_{f;\, \nu}(\phi)=\int_{\Y} \phi\, \mathrm{d}\nu$, and $\delta J^c_{f;\, \mu}(\phi)=\int_{\Y} \phi\, \mathrm{d}[-\left(\push{T_{f^c}}{\mu}\right)]$. In short, we denote by $\delta J_{f;\, \nu}=\nu$ and $\delta J_{f;\, \mu}^c=-\left(\push{T_{f^c}}{\mu}\right)$.
    \item $f$ is optimal to the dual problem \eqref{DP2} if and only if $\delta J_{f;\mu}^c+\delta J_{f;\nu}=\nu -\push{\left(T_{f^c}\right)}{\mu}=0$. Moreover, there exists a joint distribution $P\in\Pi(\mu,\nu)$ satisfying \eqref{eq:ot-optimality}. In fact, $P$ is optimal to the Kantorovich problem \eqref{KP}.
\end{enumerate}
\end{lemma}

\begin{lemma}[Theorem 1.13 and Remark 1.14 in \cite{Ambrosio2013User}]\label{lem:Kan-examples}
Assume measurable map $T$ from $\X$ to $\Y$ is in the form of \eqref{eq:map} for some continuous function $f$, that is, $T(x)=(\nabla_x c(x,\cdot))^{-1}(\nabla f^c(x))$. For every $\mu\in\Pc(\X)$ such that $c(x,y)\leqslant g(x)+f(y)$ for some $g\in L_1(\mu)$ and $f\in L_1(\push{T}{\mu})$, then $f$ is the unique maximizer to the dual problem \eqref{DP2}
    \[
    K_c(\mu,\push{T}{\mu})=\int_{\X}c(x,T(x))\, \mathrm{d}\mu(x)=\int_{\X} f^c(x)\, \mathrm{d}\mu+\int_{\X}f(T(x))\, \mathrm{d}\mu=D_{c}(\mu,\push{T}{\mu}).
    \]
Here, $f$ is uniquely defined on the region where $\push{T}{\mu}$ has positive density, up to constants.
\end{lemma}

\section{The Signed Barycenter Problem}\label{sec:primal}

Given an affine combination of weights $\left\{a_{i}\right\}_{i=1}^m\subseteq \R$ satisfying $\sum_{i=1}^m a_{i} = 1$, along with a set of probability measures $\left\{\mu_{i}\right\}_{i=1}^m\subseteq \Pc(\X)$, the barycenter functional
    \begin{equation}\label{eq:bary-func}
    \mathscr{B}(\nu) = \sum_{i=1}^m a_{i} K_{c}(\mu_{i},\nu)=\sum_{i\in I_+} a_i K_c(\mu_i,\nu) - \sum_{i\in I_-}\abs{a_i} K_c(\mu_i,\nu)
    \end{equation}
is also referred to as the \textit{primal functional}. The objective of the primal problem is to minimize $\mathscr{B}(\nu)$ in the space $\Pc(\Y)$. Any minimizer is referred to as a \textit{signed barycenter}. We assume $\left\{a_i\right\}_{i=1}^m$ are nonzero and ordered without loss of generality, i.e., $a_1\geqslant \cdots \geqslant a_\ell>0> a_{\ell+1}\geqslant\ldots \geqslant a_m$. Let $I=I_+ \cup I_{-}$, where $I_{+}=\left\{1,\ldots,\ell\right\}$ and $I_{-}=\left\{\ell+1,\ldots,m\right\}$ are the index sets for positive and negative weights respectively.

\begin{theorem}[Existence]\label{thm:existence}

Under our assumptions on the sets $\X, \Y$ and the cost $c$, $\B$ is coercive and lower semicontinuous over $\Pc(\Y)$.  As a result, there exists a signed barycenter $\bar{\nu}\in \Pc(\Y)$ that minimizes \eqref{eq:bary-func}.  
\end{theorem}
\begin{proof}
We begin by considering lower level sets $L_{\beta}:=\{\nu\in \Pc(\Y): \B(\nu) \leqslant \beta\}$ for some $\beta\in \R$. To show that $\B$ has minimizers in $\Pc(\Y)$, we will show that $L_{\beta}$ is weakly precompact and $\B$ is lower semicontinuous in with respect to weak convergence in $\P(\Y)$ (in duality with bounded continuous test functions).  Note that the main challenge here is that weak convergence in  $\P(\Y)$ does not guarantee the convergence of integrals involving the cost function due to the unboundeness of the cost.  This is overcome by using the positive coefficients to dominate the negative coefficients. 

First we establish precompactness in $\P(\Y)$. 
Choose $\varepsilon>0$ such that $a_{\varepsilon}:= \sum_{i\in I_+} (1-\varepsilon)a_i+\sum_{i\in I_-} (1+\varepsilon)a_i>0$. Given some $\nu\in L_{\beta}$, let $\pi_i\in \Pi(\mu_i,\nu)$ be an OT plan for each $i\in I$. By Assumption \ref{asp:cost_growth} on the cost, it follows from \eqref{eq:cost_epsilon_property} and \eqref{eq:cost_epsilon_property_2} that there exists a constant $B_{\varepsilon}>0$ such that
    \begin{align}
    \B(\nu)&=\sum_{i\in I_+} a_i \int_{\X\times \Y} c(x,y)\,\mathrm{d}\pi_{i}(x,y) + \sum_{i\in I_-} a_i \int_{\X\times \Y} c(x,y)\,\mathrm{d}\pi_{i}(x,y)\notag\\
     &\geqslant\sum_{i\in I_+}  a_i \int_{\X\times\Y} \big((1-\varepsilon)c(x_0,y)-B_{\varepsilon}c(x,y_0)\big)\, \mathrm{d} \pi_{i}(x,y)\notag\\
     &\hspace{3em}+\sum_{i\in I_-}  a_i \int_{\X\times\Y} \big((1+\varepsilon)c(x_0,y)+B_{\varepsilon}c(x,y_0)\big)\, \mathrm{d}\pi_{i}(x,y)\notag\\
     &\geqslant a_{\varepsilon}\int_{\Y} c(x_0,y)\,\mathrm{d}\nu(y)-B_{\varepsilon}\sum_{i=1}^m|a_i|\int_{\X}c(x,y_0)\, \mathrm{d}\mu_i(x).\label{eq:low-bnd-bary-v}
    \end{align}
This implies that there exists some constant $C_{\varepsilon}>0$ such that
    \[
    \int_{\Y} c(x_0,y)\,\mathrm{d}\nu(y)\leqslant  \frac{1}{a_{\varepsilon}}\left[ \B(\nu)+B_{\varepsilon} \sum_{i=1}^m|a_i|\int_{\X}c(x,y_0)\,\mathrm{d}\mu_i(x)\right]\leqslant C_{\varepsilon}(\beta+1),
    \]
and for each OT plan $\pi_i$, 
     \[
     \int_{\X\times \Y} (c(x,y_0)+c(x_0,y)\big)\, \mathrm{d}\pi_{i}(x,y)\leqslant  C_{\varepsilon}(\beta+1).
    \]
Thus, we see that $\B$ is coercive over $\Pc(\Y)$ and in particular, the lower level sets $L_{\beta}$ are weakly precompact in $\P(\Y)$.

Now we want to show that $\B$ is weakly lower semicontinuous over $L_{\beta}$.  Given a sequence  $\{\nu_n\} \subseteq L_{\beta}$, let $\pi_{i,n}\in \Pi(\mu_i, \nu_n)$ be an OT plan between $\mu_i$ and $\nu_n$. By the stability result in Theorem 5.20 of \cite{Villani2009Old} we may assume that $\nu_n$ and each $\pi_{i,n}$ converge weakly (in duality with continuous bounded functions) to limits $\nu\in \P(\Y)$ and $\pi_{i}\in \P(\X\times\Y)$ respectively. Moreover, $\nu\in \Pc(\Y)$ (from our discussion in Subsection \ref{sub:space}) and $\pi_i\in\Pi (\mu_i,\nu)$ is an OT plan.

Let $E\subset \X$ and $F\subset \Y$ be compact sets.  We can then split
    \begin{multline*}
    \liminf_{n\to\infty} \B(\nu_n)= \sum_{i=1}^m a_i \int_{E\times F} c(x,y)\, \mathrm{d}\pi_{i}(x,y)\\
    +\liminf_{n\to\infty} \sum_{i=1}^m a_i\Big( \int_{(\X\setminus E) \times F} c(x,y)\, \mathrm{d}\pi_{i,n}(x,y)+\int_{\X \times (\Y\setminus F)} c(x,y)\, \mathrm{d}\pi_{i,n}(x,y)\Big).
    \end{multline*}
Using \eqref{eq:cost_epsilon_property} from Assumption \ref{asp:cost_growth} of the cost, the first term in the second line satisfies 
    \[
    \begin{aligned}
    \sum_{i=1}^m a_i \int_{(\X\setminus E)\times F} c(x,y)\,\mathrm{d}\pi_{i,n}(x,y)&\geqslant -\sum_{i=2}^m \abs{a_i}\int_{(\X\setminus E)\times F} (B_1 c(x,y_0)+2c(x_0,y))\, \mathrm{d}\pi_{i,n}(x,y)\\
    &\geqslant -\sum_{i=2}^m \abs{a_i} \int_{\X\setminus E} (B_1 c(x,y_0)+\sup_{y\in F}2c(x_0,y)) \,\mathrm{d}\mu_i(x).
    \end{aligned}
    \] 
For the second term in the second line, we use the estimation trick as in \eqref{eq:low-bnd-bary-v}, obtaining
    \[
    \begin{aligned}
    &\sum_{i=1}^m a_i \int_{\X\times (\Y\setminus F)} c(x,y) \,\mathrm{d}\pi_{i,n}(x,y)\\
    \geqslant &a_{\varepsilon} \int_{\Y\setminus F} c(x_0,y)\,\mathrm{d}\nu_n(y)- B_{\varepsilon} \sum_{i=1}^m \abs{a_i} \int_{\X\times (\Y\setminus F)} c(x,y_0)\,\mathrm{d}\pi_{i,n}(x,y)\\
    \geqslant &-B_{\varepsilon}\sum_{i=1}^m \abs{a_i} \int_{\X\times (\Y\setminus F)} c(x,y_0)\,\mathrm{d}\pi_{i,n}(x,y).
    \end{aligned}
    \]
Combining the above inequalities and letting $E$ approach $\X$, it follows that
\[
\liminf_{n\to\infty} \B(\nu_n)\geqslant \sum_{i=1}^m a_i \int_{\X\times F} c(x,y)\mathrm{d}\pi_{i}(x,y) -\limsup_{n\to\infty} B_{\varepsilon}\sum_{i=1}^m|a_i|\int_{\X\times(\Y\setminus F)}c(x,y_0)\mathrm{d}\pi_{i,n}(x,y).
\]

Note that $\int_{\X\times(\Y\setminus F)}c(x,y_0)\, \mathrm{d}\pi_{i,n}(x,y)=\int_{\X}c(x,y_0)\,\mathrm{d}\tilde{\mu}^F_{i,n}(x)$
where $\tilde{\mu}_{i,n}^F\in \Pc(\X)$ is a measure such that $\tilde{\mu}_{i,n}^F(\X)=\nu_n(\Y\setminus F)$ and $\tilde{\mu}_{i,n}^F\ll \mu_i$ (in fact for any Borel set $Q\subset \X$ we have $\tilde{\mu}_{i,n}^F(Q)\leqslant \mu_{i}(Q)$).  If we define $R_{\alpha}:=\{x\in \X: c(x,y_0)>\alpha\}$, then it follows that 
\[
\limsup_{n\to\infty} B_{\varepsilon}\sum_{i=1}^m|a_i|\int_{\X}c(x,y_0)\,\mathrm{d}\tilde{\mu}^F_{i,n}(x)\leqslant B_{\varepsilon}\sum_{i=1}^m|a_i|\Big(\alpha \nu(\Y\setminus F)+\int_{R_{\alpha}} c(x,y_0)\,\mathrm{d}\mu_i(x)\Big).
\]
Allowing $F$ to approach $\Y$ and then sending $\alpha\to\infty$, it follows that 
\[
\limsup_{n\to\infty} B_{\varepsilon}\sum_{i=1}^m|a_i|\int_{\X\times(\Y\setminus F)}c(x,y_0)\mathrm{d}\pi_{i,n}(x,y)\leqslant 0,
\]
and so it follows that 
\[
\liminf_{n\to\infty} \B(\nu_n)\geqslant \sum_{i=1}^m a_i \int_{\X\times \Y} c(x,y)\, \mathrm{d}\pi_{i}(x,y)=\B(\nu).
\]
This completes the proof of lower semicontinuity.  The existence of minimizers now follows from the direct method of the calculus of variations.
\end{proof}

\subsection{Properties of the barycenter functional}

Now that we have demonstrated the existence of minimizers, we turn to the task of characterizing the minimizers by studying various notions of differentiability and  critical points.  As is standard in optimal transport, there are at least two distinct and useful notions of paths between measures.  The simplest paths are those given by convex interpolation, namely if one has measures $\nu_0, \nu_1\in \Pc(\Y)$, we can consider the interpolation $\nu_t:=t\nu_1+(1-t)\nu_0$.  On the other hand, we can also consider paths along generalized geodesics with respect to a base point in $\Pc(\Y)$.  Namely, given an abstract probability space $(\Omega, \mathcal{M}, \rho)$ with sufficiently rich structure, we can interpolate between $\nu_0$ and $\nu_1$ by finding a curve $Y:[0,1]\times\Omega\to\Y$ 
such that $t\mapsto Y(t,z)$ is $C^1$ and $\nu_t:=\push{Y(t,\cdot)}{\rho}$ matches the endpoints at $t=0,1$.   Since there is no natural choice of a distinguished basepoint for the general form of our problem, when considering generalized geodesics we will work with the abstract space $(\Omega, \mathcal{M}, \rho)$. However, we will see that when at least one of the measures $\mu_i$ is absolutely continuous, we can choose it as the basepoint and obtain a sharper characterization of differentiability properties along generalized geodesics.

We start with the derivatives of $\B$ at an absolutely continuous measure.

\begin{proposition}\label{prop:gg_derivative}
Let $\nu\in \Pc^{\textup{AC}}(\Y)$ and $w:\Y\to\R^d$ be a smooth compactly supported vector field.  Furthermore, suppose that $S:[0,1]\times\Y\to\Y$ satisfies the Lagrangian flow 
    \[
    S(t,y)=y+\int_0^t w(S(\tau,y))\,\mathrm{d}\tau
    \]
for all $t\in [0,1]$ and $y\in \Y$. If we set $\nu_t=\push{S(t,\cdot)}\nu$, then
    \[
    \lim_{t\to 0^+}\frac{\B(\nu_t)-\B(\nu)}{t}=\int_{\Y} w(y)\cdot \sum_{i=1}^m a_i \nabla_y c(X_i(y), y)\, \mathrm{d}\nu(y),
    \]
where $X_i$ is the optimal transport map from $\nu$ to $\mu_i$.
\end{proposition}
\begin{proof}
For convenience we will write $S_t(y):=S(t,y)$.
Let $f_i:\Y\to\R$ be the optimal Kantorovich potential for the transport of $\nu$ to $\mu_i$. We then have the standard inequalities from the classical OT theory
    \[
    \int_{\Y} f_i(S_t(y))\,\mathrm{d}\nu(y)+\int_{\X}f_i^c(x_i)\,\mathrm{d}\mu_i(x_i)\leqslant K_c(\mu_i, \nu_t)\leqslant \int_{\Y} c(X_i(y), S_t(y))\,\mathrm{d}\nu(y),
    \]
with equality when $t=0$. 
As a result, we can bound
    \[
    \frac{\B(\nu_t)-\B(\nu)}{t}\leqslant
    \int_{\Y}\Big( \sum_{i\in I_+} a_i\frac{c(X_i(y), S_t(y))-c(X_i(y), y)}{t}+\sum_{i\in I_-}a_i\frac{f_i(S_t(y))-f_i(y)}{t}\Big)\mathrm{d}\nu(y).
    \]
Since $f_i$ is locally Lipschitz, and $w$ is a smooth compactly supported vector field, it follows by dominated convergence that
    \[
    \lim_{t\to 0^+} \frac{\B(\nu_t)-\B(\nu)}{t}\leqslant \int_{\Y} w(y)\cdot \Big(\sum_{i\in I_+} a_i \nabla_y c(X_i(y), y)+\sum_{i\in I_-} a_i \nabla f_i(y)\, \Big)\,\mathrm{d}\nu(y).
    \]
Thanks to the absolute continuity of $\nu$ and the twisted condition from Assumption \eqref{asp:cost_twist} on the cost, for $\nu$-almost every $y$, the optimal map $X_i$ is determined by the equation $\nabla f_i(y)=\nabla_y c(X_i(y),y)$.  Thus, the above equation is equivalent to
    \[
    \lim_{t\to 0^+}\frac{\B(\nu_t)-\B(\nu)}{t}\leqslant \int_{\Y} w(y)\cdot \sum_{i=1}^m a_i \nabla_y c(X_i(y), y)\, \mathrm{d}\nu(y).
    \]
The opposite inequality follows from an identical argument.
\end{proof}

The signed barycenter need not be absolutely continuous, even when all input measures are absolutely continuous. This distinguishes from the classical barycenter problem \citep[see Theorem 5.1]{Agueh2011Barycenters}. Remark 3.5 in \cite{tornabene2024generalizedwassersteinbarycenters} constructs an example in which the signed barycenter of two Gaussians supported on $\R$ is the Dirac measure. This result can be generalized to multiple marginals and general cost.

Consequently, it is necessary to understand derivatives of $\B$ at general points in $\Pc(\Y)$. While we cannot guarantee differentiability, Proposition \ref{prop:gg_derivative_bound} gives us some useful upper and lower slope bounds. The absence of absolute continuity requires us to work with couplings and their disintegrations, instead of maps in Proposition \ref{prop:gg_derivative}. 

\begin{definition}\label{def:opt-coup}
    Let $\nu\in \Pc(\Y)$, we define the space $\Pi_*(\mu_1,\ldots, \mu_m;\nu)\subseteq \Pi(\mu_1,\ldots, \mu_m,\nu)$ as the set of couplings such that for any $\gamma\in \Pi_*(\mu_1,\ldots, \mu_m;\nu)$ and any $i\in \{1,\ldots, m\}$ we have 
    \[
    K_c(\mu_i,\nu)=\int_{\X^m\times \Y} c(x_i,y)\, \mathrm{d}\gamma(x_1,\ldots, x_m,y).
    \]
\end{definition}
The space $\Pi_*(\mu_1,\cdots,\mu_m; \nu)$ is not empty, due to the generalized gluing lemma or iterated Dudley's lemma \citep[Proposition 8.6]{Ambrosio2021Lectures}.

\begin{proposition}\label{prop:gg_derivative_bound}
Given $\nu\in \Pc(\Y)$ and $\rho\in \Pc^{\textup{AC}}(\Omega)$ for some open set $\Omega\subset \R^d$, let $Y:[0,1]\times\Omega\to \Y$ be a curve $Y_t(z):=Y(t,z)$  satisfying $\push{Y_0}{\rho}=\nu$.  If $t\mapsto Y(t,z)$ is  uniformly $C^1$ for all $z\in \Omega$, then for $\nu_t:=\push{Y_t}{\rho}$, there exist vector fields $V_+, V_-:\Omega\to \R^d$ (potentially depending on the curve $Y_t$) such that
\[
\int_{\Omega}\partial_t Y(0,z)\cdot V_-(z)\, \mathrm{d}\rho(z)\leqslant \liminf_{t\to 0^+}\frac{\B(\nu_t)-\B(\nu)}{t}
\]
and
\[
\limsup_{t\to 0^+}\frac{\B(\nu_t)-\B(\nu)}{t}\leqslant \int_{\Omega}\partial_t Y(0,z)\cdot V_+(z)\, \mathrm{d}\rho(z).
\]
In particular,
for each $i\in \{1,\ldots, m\}$, there exists a map $X_i:\Omega\to\X$,  couplings $\pi_i^{\pm}\in \Pi(\mu_i,\rho)$, and disintegrations $\pi_{i,z}^{\pm}$ of $\pi_i^{\pm}$ with respect to $\rho$ such that
the  vector fields $V_+, V_-$ are given by
\begin{align*}
    V_+(z)&= \sum_{i\in I_+}a_i\nabla_y c(X_i(z), Y_0(z))+\sum_{i\in I_-}a_i\int_{\X} \nabla_y c(x_i,Y_0(z))\,\mathrm{d}\pi_{i,z}^+(x_i),\\
    V_-(z)&=\sum_{i\in I_+}a_i\int_{\X} \nabla_y c(x_i,Y_0(z))\,\mathrm{d}\pi_{i,z}^{-}(x_i)+\sum_{i\in I_-}a_i\nabla_y c(X_i(z), Y_0(z)).
\end{align*}

\end{proposition}
\begin{remark}
    Note that the potential dependence of $V_+, V_-$ on the curve $Y_t$ is  through the disintegrations $\pi_{i,z}$, whereas the maps $X_i$ can be chosen independently of the curve.
\end{remark}
\begin{proof}
For each $t\in [0,1]$ there exists a plan  $\gamma_t\in \Pi_*(\mu_1,\ldots, \mu_m;\nu_t)$ and measurable maps $X_{1,t},\ldots, X_{m,t}:\Omega\to\X$ such that $\push{(X_{1,t},\ldots, X_{m,t},Y_t)}{\rho}=\gamma_t$.  Since for any $t, s\in [0,1]$ we have $\push{(X_{1,s},\ldots,X_{m,s}, Y_t)}{\rho}\in \Pi(\mu_1,\ldots, \mu_m,\nu_t)$, it follows that  
\[
K_c(\mu_i, \nu_t)= \int_{\Omega} c(X_{i,t}(z), Y_t(z))\,\mathrm{d}\rho(z)\leqslant \int_{\Omega} c(X_{i,s}(z), Y_t(z))\,\mathrm{d}\rho(z).
\]
For $i\in I_+$, we have the upper and lower bounds for $ a_i (K_c(\mu_i,\nu_t)-K_c(\mu_i,\nu))$
    \begin{multline*}
    a_i \int_{\Omega} c(X_{i,t}(z),Y_t(z))\,\mathrm{d}\rho(z)- a_i \int_{\Omega} c(X_{i,t}(z),Y_0(z))\,\mathrm{d}\rho(z)\leqslant a_i K_c(\mu_i,\nu_t)- a_i K_c(\mu_i,\nu)\\
    \leqslant a_i\int_{\Omega } c(X_{i,0}(z),Y_t(z))\,\mathrm{d}\rho(z)- a_i \int_{\Omega}c(X_{i,0}(z),Y_0(z))\, \mathrm{d}\rho(z).
    \end{multline*}
Similarly, we also have the upper and lower bounds when $i\in I_{-}$ as
    \begin{multline*}
        a_i \int_{\Omega} c(X_{i,0}(z),Y_t(z))\,\mathrm{d}\rho(z)-a_i \int_{\Omega} c(X_{i,0}(z),Y_0(z))\,\mathrm{d}\rho(z)\leqslant a_i K_c(\mu_i,\nu_t) - a_i K_c(\mu_i, \nu)\\
        \leqslant a_i\int_{\Omega} c(X_{i,t}(z),Y_t(z))\,\mathrm{d}\rho(z)-a_i \int_{\Omega}c(X_{i,t}(z),Y_0(z))\,\mathrm{d}\rho(z).
    \end{multline*}
This implies that the difference quotient $\frac{\B(\nu_t)-\B(\nu)}{t}$ is bounded from above by
\[
\int_{\Omega}\Big( \sum_{i\in I_+} a_i\frac{c(X_{i,0}, Y_t)-c(X_{i,0},Y_0)}{t}+\sum_{i\in I_-} a_i \frac{c(X_{i,t}, Y_t)-c(X_{i,t},Y_0)}{t}\Big)\,\mathrm{d}\rho,
\]
and bounded from below by
\[
\int_{\Omega}\Big( \sum_{i\in I_+} a_i\frac{c(X_{i,t}, Y_t)-c(X_{i,t},Y_0)}{t}+\sum_{i\in I_-} a_i \frac{c(X_{i,0}, Y_t)-c(X_{i,0},Y_0)}{t}\Big)\,\mathrm{d}\rho.
\]

Let $\pi_{i,t}=\push{(X_{i,t},\id)}\rho\in \Pi(\mu_i, \rho)$ and note that there must exist subsequences $t_n^+, t_n^-\to 0$ and measures $\pi_i^+, \pi_i^-\in \Pi(\mu_i,\rho)$
such that 
    \[
    \begin{aligned}
    \lim_{n\to\infty}\frac{\B(\nu_{t_n^+})-\B(\nu)}{t_n^+}&=\limsup_{t\to 0^+} \frac{\B(\nu_t)-\B(\nu)}{t},\\
    \lim_{n\to\infty}\frac{\B(\nu_{t_n^-})-\B(\nu)}{t_n^-}&=\liminf_{t\to 0^+} \frac{\B(\nu_t)-\B(\nu)}{t},
    \end{aligned}
    \]
and $\pi_{i,t_n^+},\pi_{i,t_n^-}$ converge weakly to $\pi_i^+, \pi_i^-$ respectively.  We then disintegrate $\mathrm{d}\pi_i^{\pm}(x_i,z)=\mathrm{d}\pi^{\pm}_{i,z}(x_i) \mathrm{d}\rho(z)$.   It now follows that $\lim_{t\to 0^+}\frac{\B(\nu_t)-\B(\nu)}{t}$ is bounded from above by
\[
\int_{\Omega}\partial_t Y_0(z)\cdot\left[\sum_{i\in I_+} a_i \nabla_y c(X_{i,0}(z), Y_0(z)) + \sum_{i\in I_-} a_i \int_{\X} \nabla_y c(x_i,Y_0(z))\,\mathrm{d}\pi^+_{i,z} (x_i)\right] \,\mathrm{d}\rho(z),
\] 
and bounded from below by
\[
\int_{\Omega}\partial_t Y_0(z)\cdot\left[\sum_{i\in I_+} a_i \int_{\X} \nabla_y c(x_i,Y_0(z))\,\mathrm{d}\pi^-_{i,z}(x_i)+\sum_{i\in I_-} a_i \nabla_y c(X_{i,0}(z), Y_0(z))\right]\,\mathrm{d}\rho(z), 
\]
which is the desired result.
\end{proof}

\begin{corollary}
Suppose $\mu_j\in \Pc^{\textup{AC}}(\X)$ for some $j\in I_+$, consider the setup in Proposition \ref{prop:gg_derivative_bound} and take $(\Omega,\rho)=(\X,\mu_j)$. If there exists a family of $c$-concave functions $\phi_t:\X\to \R$ for all $t\in [0,1]$, such that for $\mu_j$-almost every $x$, the curve $Y_t:\X\to \Y$ defined in Proposition \ref{prop:gg_derivative_bound} satisfies the equation  
    \[
    \nabla_x c(x,Y_t(x))=\nabla \phi_t(x),
    \]
then
    \[
    \int_{\X}\partial_t Y(0,x)\cdot V_-(x)\, \mathrm{d}\mu_j(x)\leqslant \lim_{t\to 0^+}\frac{\B(\nu_t)-\B(\nu)}{t}
    \]
where $V_-$ satisfies the following refined formula
    \begin{equation*}
          V_-(x)=a_j\nabla_y c(x, Y_0(x))+ \sum_{i\in I_+\setminus\{j\}}a_i\int \nabla_y c(x_i,Y_0(x))\,\mathrm{d}\pi_{i,x}^-(x_i) + \sum_{i\in I_-}a_i\nabla_y c(X_i(x), Y_0(x)) .
    \end{equation*}
An analogous refinement for $V_+$ holds if we have $j\in I_-$ instead of $I_+$. 
\end{corollary}
\begin{remark}
    If $I_+=\{j\}$, then $V_-$ is independent of the curve.  
\end{remark}
\begin{proof}
    By our main assumption, it follows that $Y_t$ is the optimal transport map from $\mu_j$ to $\nu_t$.  As a result, we have \[
   \frac{ K_c(\mu_j, \nu_t)-K_c(\mu_j,\nu)}{t}=\int_{\X} \frac{c(x,Y_t(x))-c(x,Y_0(x))}{t}\, \mathrm{d}\mu_j(x).
    \] 
    Hence, in the argument of Proposition \ref{prop:gg_derivative_bound}, we can treat $j$ as we treat the indices in $I_-$  when bounding the difference quotient $\frac{\B(\nu_t)-\B(\nu)}{t}$ from below. The rest of the argument is identical to Proposition \ref{prop:gg_derivative_bound}.
\end{proof}

Finally, we establish a congruence condition analogous to that arising in the classical barycenter problem. Here we apply \eqref{DP2} to each term in $\B$. As we will see later, this constraint will lead to the new dual formulation presented in Section \ref{sec:dual}.

\begin{proposition}\label{prop:dual-sum-zero}
If $\bar{\nu}\in \Pc(\Y)$ is a local minimizer of $\B$, then there exist optimal $c$-concave Kantorovich potentials $\{f_i\}_{i=1}^m$, where $f_i:\Y\to\R$ for the transport of $\bar{\nu}$ to $\mu_i$ such that $\displaystyle \sum_{i=1}^m a_i f_{i}(y)=\inf_{\tilde{y}\in \Y} \sum_{i=1}^m a_if_{i}(\tilde{y})=0$
for $\bar{\nu}$-almost every $y\in \Y$.
\end{proposition}
\begin{proof}
Given some $\nu\in \Pc(\Y)$, set $\nu_t:=(1-t)\bar{\nu}+t\nu$ for $t\in [0,1]$. For each $i$ and $t\in [0,1]$, let $f_{i,t}:\Y\to\R$ be an optimal $c$-concave Kantorovich potential for the transport of $\nu_t$ to $\mu_i$. After possibly adjusting each $f_{i,t}$ by a constant, we may assume that there exists a point $y_0\in \textup{spt}(\bar{\nu})$ such that $f_{i,t}(y_0)=0$ for each $i$.

Since $\left\{f_{i,t}\right\}$ are locally Lipschitz, there exists a subsequence $t_n\to 0$ such that $f_{i,t_n}$ converges uniformly to a limit $f_i$ 
on compact subsets of $\Y$.
Note that the limits $f_i$ are necessarily optimal $c$-concave Kantorovich potentials for the transport of $\bar{\nu}$ to $\mu_i$.

Since $\bar{\nu}$ is a local minimizer, it follows that 
    \[
    0\leqslant \lim_{n\to\infty} \frac{\B(\nu_{t_n})-\B(\bar{\nu})}{t_n}.
    \]
On the other hand, Kantorovich duality gives us
    \[
    \frac{\B(\nu_{t_n})-\B(\bar{\nu})}{t_n}\leqslant \sum_{i\in I_+}a_i \int_{\Y} f_{i,t_n}(y)\,\mathrm{d}(\nu-\bar{\nu})(y)+\sum_{i\in I_-}a_i \int_{\Y} f_{i}(y)\, \mathrm{d}(\nu-\bar{\nu})(y)
    \]
Sending $n\to\infty$ and using the convergence of the $f_{i,t_n}$ and the previous inequality, it follows that 
    \[
    0\leqslant  \int_{\Y} \sum_{i=1}^m a_if_{i}(y) \,\mathrm{d}(\nu-\bar{\nu})(y)
    \]
Hence, 
\[
 \sum_{i=1}^m a_i f_{i}(y)=\inf_{\tilde{y}\in \Y} \sum_{i=1}^m a_if_{i}(\tilde{y})
\]
for $\bar{\nu}$-almost every $y\in \Y$.  Since we have $f_{i}(y_0)=0$ for each $i$, it follows that the infimum above has value 0.
\end{proof}

\subsection{The special case of only one positive weight}\label{sub:one-pos}

We now turn our attention to the special case where only one of the weights is positive, i.e., $I_+=\{1\}$. In this case we can guarantee convexity and uniqueness properties of the signed barycenter problem, as long as the cost function satisfies a certain convexity condition. When $I_+=\{1\}$ and the cost is quadratic, this special case has already been noted in prior work \citep{Gallouet2025Metric, tornabene2024generalizedwassersteinbarycenters}. As explained in \cite{Gallouet2025Metric}, an advantage in the 2-Wasserstein metric is that $\B(\nu)$ is 1-convex along generalized geodesics with base measure $\mu_1$. This follows from \citet[Lemma 9.2.1 and Proposition 9.3.12]{Ambrosio2008Gradient}, which state that $\frac{a_1}{2} W_2^2(\cdot,\mu_1)$ is $a_1$-convex along generalized geodesics with base measure $\mu_1$ and $-\frac{\abs{a_i}}{2}W_2^2(\cdot,\mu_i)$ is $(-\abs{a_i})$-convex along generalized geodesic with any base measure.

\begin{theorem}\label{thm:convex-unique}
Let $I_+=\{1\}$ in the barycenter functional \eqref{eq:bary-func}. If there exists a continuous function $\theta:\X\times\Y^2\times [0,1]\to\Y$ such that   $\theta(x_1, y_0,y_1, 0)=y_0$,  $\theta(x_1, y_0,y_1, 1)=y_1$, and for all $(x_1,\ldots, x_m, y_0,y_1)\in \X^{m}\times\Y^2$ with $y_0\neq y_1$, 
    \begin{equation}\label{eq:cost_curve_convex}
    t\mapsto \sum_{i=1}^m a_i c\big(x_i, \theta(x_1, y_0,y_1, t)\big)
    \end{equation}
is (strictly) convex, then for any $\nu_0, \nu_1\in \Pc(\Y)$ there exists an explicit continuous interpolating curve $\nu_t$ such that
    \[
    t\mapsto \B(\nu_t)
    \]
is (strictly) convex. Consequently, if $t\mapsto \B(\nu_t)$ is strictly convex, then the signed barycenter is unique.
\end{theorem}
\begin{remark}

When $c$ is the quadratic cost or other costs based on the Euclidean distance, it is natural to choose the standard convex interpolation $\theta(x_1, y_0,y_1, t)=ty_1+(1-t)y_0$. In particular, this choice for the quadratic cost leads to the fact that $t\mapsto \B(\nu_t)$ is strictly convex. This recovers the generalized geodesic convexity  (with respect to the basepoint $\mu_1$) and the uniqueness results from \cite{tornabene2024generalizedwassersteinbarycenters}.

\end{remark}
\begin{proof}
For $j=0,1$, let $\pi^j\in \Pi(\mu_1, \nu_j)$ be the optimal plan transporting $\mu_1$ to $\nu_j$.  By the Gluing Lemma there exists a measure $\gamma\in \Pi(\mu_1,\nu_0, \nu_1)$ such that $P^{0}_{\#}\gamma=\pi^0$ and $P^{1}_{\#}\gamma=\pi^1$ where $P^{j}:\X\times\Y^2\to\X\times\Y$ is the projection map $P^j(x_1,y_0,y_1)=(x_1,y_j)$ for $j\in \{0,1\}$.  We then define two maps $Q_{\theta, t}:\X\times\Y^2\to\Y, \tilde{Q}_{\theta, t}:\X^2\times\Y^2\to\X\times\Y$
    \[Q_{\theta,t}(x_1,y_0,y_1)=\theta(x_1, y_0,y_1, t)\qquad\textrm{and}\qquad \tilde{Q}_{\theta, t}(x_i,x_1,y_0,y_1)=(x_i,\theta(x_1, y_0,y_1, t)),\]
and use them to define pushforward measures. 

Define $\nu_t:=Q_{\theta, t\, \#}\gamma$. Note that $\tilde{Q}_{\theta, t}(x_i,x_1,y_0,y_1)= (x_i,Q_{\theta,t}(x_1,y_0,y_1))$. Then every coupling $\tilde{\pi}\in \Pi(\mu_i, \nu_t)$ can be represented as $\tilde{\pi}=\tilde{Q}_{\theta,t\, \#}\Gamma_i$ for some $\Gamma_i\in \Pi(\mu_i,\gamma)$. Therefore, for any $i\in \{2,\ldots, m\}$
\[
K_c(\mu_i,\nu_t)=\inf_{\Gamma_i\in \Pi(\mu_i, \gamma)} \int_{\X^2\times\Y^2} c(x_i, \theta(x_1, y_0,y_1, t))\,\mathrm{d}\Gamma_i(x_i, x_1, y_0,y_1).
\]
Since we can glue each of these measures along $\gamma$, it follows that 
\begin{equation}\label{eq:an_negative_sum_formula}
\sum_{i=2}^m |a_i|K_c(\mu_i,\nu_t)=\inf_{\Gamma\in \Pi(\mu_2,\ldots, \mu_m, \gamma)} \int_{X^m\times\Y^2} \sum_{i=2}^m |a_i| c\big(x_i, \theta(x_1, y_0,y_1, t)\big)\, \mathrm{d}\Gamma. 
\end{equation}
Next, we observe that
\begin{equation}\label{eq:an_one_upper}
K_c(\mu_1,\nu_t)\leqslant \int_{\X\times\Y^2} c\big( x_1,\theta(x_1, y_0,y_1, t)\big)\,\mathrm{d}\gamma( x_1, y_0,y_1)
\end{equation}
Thus, combining \eqref{eq:an_negative_sum_formula} and \eqref{eq:an_one_upper} and the convexity of the curve $\theta$,  we obtain the inequality
    \[
    \begin{aligned}
    &\B(\nu_t)=a_1 K_c(\mu_1,\nu_t) - \sum_{i=2}^m \abs{a_i} K_c(\mu_i,\nu_t)\\
    \leqslant& \sup_{\Gamma\in \Pi(\mu_2,\ldots, \mu_m, \gamma)} \int_{\X^m\times\Y^2} \sum_{i=1}^m a_i c\big(x_i, \theta(x_1, y_0,y_1, t)\big)\,\mathrm{d}\Gamma(x_2,\ldots, x_m, x_1, y_0, y_1)\\
    \leqslant& \sup_{\Gamma\in \Pi(\mu_2,\ldots, \mu_m, \gamma)} \int_{\X^m\times\Y^2} \Big[t\sum_{i=1}^m a_i c\big(x_i, y_1\big)+(1-t)\sum_{i=1}^m a_i c\big(x_i, y_0\big)\Big]\,\mathrm{d}\Gamma\\
    \leqslant& \sum_{j=0}^1 \lambda_j \Big( \int_{\X\times\Y^2} a_1 c(x_1,y_j)\,\mathrm{d}\gamma+\sup_{\Gamma^j_i\in \Pi(\mu_i,\gamma)} \sum_{i=2}^m a_i\int_{\X^2\times\Y^2} c\big(x_i, y_j)\, \mathrm{d}\Gamma_i^j(x_i,x_1,y_0,y_1)\Big),
    \end{aligned}
    \]
where we set $\lambda_0:=(1-t)$ and $\lambda_1:=t$ for notational compactness.  Recalling that $P^{j}_{\#}\gamma=\pi^j$ is the optimal transport plan between $\mu_1$ and $\nu_j$, it follows that the previous line is equal to 
    \[
    t\B(\nu_1)+(1-t)\B(\nu_0),
    \]
and hence we have shown that $\B(\nu_t)\leqslant t\B(\nu_1)+(1-t)\B(\nu_0)$.

Finally, if (\ref{eq:cost_curve_convex}) is strictly convex, there exists a continuous nonnegative function $\omega:\X^m\times\Y^2\times (0,1)\to [0,\infty)$ such that 
\[
\sum_{i=1}^m a_i c\big(x_i, \theta(x_1, y_0,y_1, t)\big)+\omega(x_1,\ldots, x_m,y_0,y_1,t)=t\sum_{i=1}^m a_i c\big(x_i, y_1\big)+(1-t)\sum_{i=1}^m a_i c\big(x_i, y_0\big)
\]
and $\omega$ is strictly positive for $t\in [0,1]$ and $y_0\neq y_1$. Let $\mathcal{D}=\{(y_0,y_1)\in \Y: y_0=y_1\}$. Choose compact sets $A_i\subset \X$ such that $\mu_i(A_i)>1/2$ and a compact set $B\subset \Y^2\setminus \mathcal{D}$ such that $\gamma(B)>0$.  Set $A=\prod_{i=2}^m A_i$.  Then for any $\Gamma\in \Pi(\mu_2,\ldots, \mu_m, \gamma)$, it follows that $\Gamma(A\times B)\geq 2^{-m+1}\gamma(B)$. Meanwhile, on $A\times B$ there must exist $\delta_t>0$ such that $\omega\geq \delta_t$ on $A\times B\times \{t\}$.  Hence, after retracing the steps of the previous argument, it follows that
 \[
 \B(\nu_t)\leqslant t\B(\nu_1)+(1-t)\B(\nu_0)-\delta_t 2^{-m+1}\gamma(B) <t\B(\nu_1)+(1-t)\B(\nu_0),
 \]
so $t\mapsto \B(\nu_t)$ is strictly convex.
\end{proof}

\section{The Dual Problem}\label{sec:dual}

In this section, we develop a new dual formulation for the signed barycenter problem \eqref{eq:bary-func}, distinct from the multimarginal dual formulations studied in \citep{Agueh2011Barycenters, Gangbo1998Multidimensional}. Similar to the classical OT, the dual problem explores the optimization over the space of functions, which is typically more tractable \cite{Jacobs2020BF}. Unlike the constraint from the multimarginal setting that complicates the explicit characterization of gradients of its dual functional \cite{Zhou2022Efficient}, we deal with the congruence condition from Proposition \ref{prop:dual-sum-zero}. 

In Subsection \ref{sub:one-pos}, we discussed the special case where only one of the weights is positive. We will assume throughout this section there are at least two positive weights, i.e., $I_+\setminus \{1\}\neq \emptyset$. We also assume that at least one marginal associated with a \textit{positive weight} is absolutely continuous, and let $\mu_1\in \Pc^{\textup{AC}}(\X)$ without loss of generality. 

Each term in \eqref{eq:bary-func} has its own dual formulation from Subsection \ref{sub:OT}:
    \[
    \begin{aligned}
    &a_i K_c(\mu_i,\nu)=\sup_{f_i}a_i[\int_{\X} f^c_i\, \mathrm{d}\mu_i+\int_{\Y} f_i\, \mathrm{d}\nu],\quad \textup{for~} i\in I_+;\\
    &a_i K_c(\mu_i,\nu)=\inf_{f_i}a_i [\int_{\X} f^c_i\, \mathrm{d}\mu_i + \int_{\Y} f_i\, \mathrm{d}\nu],\quad \textup{for~} i\in I_-. 
    \end{aligned}
    \]
Here the optimizations are over $c$-concave functions. For the remainder of this paper, all dual formulations are understood to be restricted to $c$-concave functions. We refer to the potentials $f_i$ for $i\in I_+$ as maximizing variables, and for $i\in I_-$ as minimizing variables. We further denote the tuple of dual potentials by $f_{[I]}=(f_{1},\ldots,f_{m})$. 

By Proposition \ref{prop:dual-sum-zero}, these Kantorovich potentials satisfy a congruence condition. Now, we dualize all terms except $a_1 K_c(\mu_1,\nu)$ in \eqref{eq:bary-func} to tackle this constraint. 
    \begin{equation}\label{eq:dual-replace-po-f}
    \mathscr{B}(\nu)=\inf_{f_{[I_-]}}\sup_{f_{[I_+\setminus \{1\}]}} \sum_{i=2}^m \int_{\X} a_i f^c_i\, \mathrm{d}\mu_i + \int_{\Y} (\sum_{i=2}^m a_{i} f_{i})\, \mathrm{d}\nu + a_1 K_c(\mu_1,\nu).
    \end{equation}
The order of $\sup$ and $\inf$ are interchangeable due to the separability. 

\begin{lemma}\label{lem:dual}
For any $\mu\in \Pc^{\textup{AC}}(\X)$, and any bounded continuous function $\phi$, we have
    \begin{equation*}
    \inf_{\nu\in \Pc(\Y)} \int_{\Y} \phi\, \mathrm{d}\nu +  K_c(\mu,\nu)=\int_{\X} (-\phi)^{ c}\, \mathrm{d}\mu.
    \end{equation*}
Furthermore, the minimizer is given by $\bar{\nu}=\push{T_{(-\phi)^c}}{\mu}$.
\end{lemma}
\begin{proof}
    \begin{align*}
        \inf_{\nu\in \Pc(\Y)}\int_{\Y} \phi\, \mathrm{d}\nu + K_c(\mu,\nu)=&\inf_T \int_{\Y}\phi\, \mathrm{d}[\push{T}{\mu}]+\int_{\X} c(x,T(x))\, \mathrm{d}\mu\\
        =&\inf_T \int_{\X} [\phi(T(x))+ c(x,T(x))]\, \mathrm{d}\mu \\
        =&\int_{\X} \inf_{y=T(x)} [\phi(y)+ c(x,y)]\, \mathrm{d}\mu=\int_{\X} (-\phi)^{c}(x)\, \mathrm{d}\mu.
    \end{align*}
    
The existence of minimizer is obtained from (see for example, Proposition 2.9 in \cite{Gangbo2004introduction}). See also part 6 in Lemma \ref{lem:c-tran-prop}.

\end{proof}

Thanks to Lemma \ref{lem:dual} and the scaling law in Lemma \ref{lem:c-tran-prop}, considering interchanging inf w.r.t $\nu$ and optimizing w.r.t dual variables $f_{i}$, we obtain from \eqref{eq:dual-replace-po-f} 
    \[
    \inf_{\nu}\mathscr{B}(\nu)\geqslant \inf_{f_{[I_-]}}\sup_{f_{[I_+\setminus \{1\}]}} \sum_{i=2}^m \int_{\X} a_i f^c_i\, \mathrm{d}\mu_i + a_1\int_{\X} (-\sum_{i=2}^m \frac{a_{i}}{a_1} f_{i})^c \, \mathrm{d}\mu_1.
    \]
Thus, we introduce the \textit{dual functional} by
    \begin{equation}\label{eq:dual-func-po-f}  
    \D(f_{[I \setminus \{1\}]})=\sum_{i=2}^m \int_{\X} a_i f^c_i\, \mathrm{d}\mu_i + a_1\int_{\X} (-\sum_{i=2}^m \frac{a_{i}}{a_1} f_{i})^c \, \mathrm{d}\mu_1.
    \end{equation}
By introducing a redundant potential 
    \begin{equation}\label{eq:redundant}
    f_1 = -\sum_{i=2}^m \frac{a_{i}}{a_1}f_{i},
    \end{equation}
owing to the symmetry of $\D$ among all potentials, we shall, for notational convenience, write $\D(f)=\sum_{i=1}^m \int_{\X} a_i f_i^c\, \mathrm{d}\mu_i$, even though the input is an $(m-1)$-tuple of dual potentials. 

To study the signed barycenter problem from this dual problem, the first concern is whether the strong duality
\begin{equation}\label{eq:duality}
    \inf_{\nu\in \Pc(\Y)} \mathscr{B}(\nu)=\inf_{f_{[I_-]}}\sup_{f_{[I_+\setminus \{1\}]}}\D(f_{[I \setminus \{1\}]})
    \end{equation}
holds or not. Crucially at this stage, the maximization and minimization in the dual problem cannot be interchanged, as $\D$ is no longer separable. 

It would be ideal if the dual functional $\D$ has the convex-concave property. While it is straightforward to verify that $\D$ is concave w.r.t maximizing variables by Lemma \ref{lem:OT-dual-optimality} and Proposition \ref{prop:func-dual-concave}, $\D$ is in general not convex w.r.t minimizing variables, noting for any $i\in I_-$, $f_i\mapsto -\abs{a_i}\int f_i^c \mathrm{d}\mu_i$ is convex whereas $f_i\mapsto a_1 \int (-\sum_{j=2}^{m}\frac{a_{j}}{a_1}f_{j})^c \mathrm{d}\mu_1$ is concave. To overcome this difficulty, we plan to characterize stationary points and saddle points to the dual functional in Subsection \ref{subsec:stationary}.

\begin{proposition}\label{prop:func-dual-concave}
$f_{[I_+\setminus \{1\}]}=(f_2,\ldots, f_{\ell})\mapsto \int_{\X} (-\sum_{i=2}^m \frac{a_{i}}{a_1}f_i)^c\, \mathrm{d}\mu_1$ is jointly concave.
\end{proposition}
\begin{proof}

For simplicity of notation, let $J_{\mu_1}^c (f)=\int_{\X} f^c \mathrm{d}\mu_1$ and note $f\mapsto J_{\mu_1}^c (f)$ is concave due to Lemma \ref{lem:OT-dual-optimality}. Given two tuples of dual potentials $(f_{[I_+\setminus \{1\}]}^{(1)},f_{[I_-]}), (f_{[I_+\setminus \{1\}]}^{(2)},f_{[I_-]})$,
    \[
    \begin{aligned}
    &J_{\mu_1}^c(-\sum_{i\in I_+\setminus \{1\}} \frac{a_i}{a_1}(t f_i^{(1)}+(1-t)f_i^{(2)}) - \sum_{i\in I_-}\frac{a_i}{a_1} f_i)\\
    =&J_{\mu_1}^c (t [-\sum_{i\in I_+\setminus \{1\}} \frac{a_i}{a_1}f_i^{(1)} - \sum_{i\in I_-} \frac{a_i}{a_1}f_i] + (1-t) [-\sum_{i\in I_+\setminus \{1\}} \frac{a_i}{a_1}f_i^{(2)} - \sum_{i\in I_-} \frac{a_i}{a_1}f_i])\\
    \geqslant & t J_{\mu_1}^c (-\sum_{i\in I_+\setminus \{1\}} \frac{a_i}{a_1}f_i^{(1)} - \sum_{i\in I_-} \frac{a_i}{a_1}f_i) + (1-t) J_{\mu_1}^c (-\sum_{i\in I_+\setminus \{1\}} \frac{a_i}{a_1}f_i^{(2)} - \sum_{i\in I_-} \frac{a_i}{a_1}f_i),
    \end{aligned}
    \]
which shows $(f_2,\ldots, f_{\ell})\mapsto J_{\mu_1}^c(-\sum_{i=2}^m \frac{a_{i}}{a_1}f_i)$ is jointly concave.

\end{proof}

The second concern is how to recover the signed barycenter from the dual problem. Suppose $\bar{f}_{[I\setminus \{1\}]}$ are optimal (i.e., a saddle point) to the dual problem, for readers familiar with the classical barycenter problem, it is natural to conjecture that
    \begin{equation}\label{eq:bary-po-f}
        \bar{\nu} = \push{\left(T_{(-\sum_{i=2}^m \frac{a_{i}}{a_1}\bar{f}_{i})^c}\right)}{\mu_1}=\push{\left(T_{\bar{f}_1^c}\right)}{\mu_1},
    \end{equation}
This observation is validated in Theorem \ref{thm:main-induced} in Subsection \ref{subsec:global}. We also discuss the uniqueness of the signed barycenter, from the perspective of the dual approach. 

Lastly, it is more practical to obtain a stationary point to the minimax problem. We propose in Subsection \ref{subsec:upgrade} a criterion to test if a stationary point is indeed a saddle point. This criterion is in a simple form for the quadratic cost $c(x,y)=\frac{1}{2}\abs{x-y}^2$.

\subsection{Characterization of stationary points}\label{subsec:stationary}

Given a point $\bar{f}_{[I\setminus \{1\}]}=(\bar{f}_{[I_+\setminus \{1\}]},\bar{f}_{[I_-]})$, we say it is a \textit{stationary point} to $\D$ if for any tuples of bounded continuous functions $(f_{[I_+\setminus \{1\}]}, f_{[I_-]})$, it holds
    \[
    \begin{aligned}
    &\lim_{\varepsilon\to 0}\dfrac{\D(\bar{f}_{[I_+\setminus \{1\}]}+\varepsilon f_{[I_+\setminus \{1\}]},\bar{f}_{[I_-]})-\D(\bar{f}_{[I_+\setminus \{1\}]},\bar{f}_{[I_-]})}{\varepsilon}\leqslant 0;\\
    &\lim_{\varepsilon\to 0}\dfrac{\D(\bar{f}_{[I_+\setminus \{1\}]},\bar{f}_{[I_-]}+\varepsilon f_{[I_-]})-\D(\bar{f}_{[I_+\setminus \{1\}]},\bar{f}_{[I_-]})}{\varepsilon}\geqslant 0.
    \end{aligned}
    \]
In addition, assuming all marginals are absolutely continuous, the stationary point $\bar{f}_{[I\setminus \{1\}]}$ can be readily characterized by the first variation from Lemma \ref{lem:Gangbo} and Lemma \ref{lem:OT-dual-optimality} as in the optimal transport problem, for any $i\in I\setminus \{1\}$ it satisfies that
    \begin{equation}\label{eq:stationary-by-pf-po}
    \push{\left(T_{\bar{f}^c_{i}}\right)}{\mu_i}=\push{\left(T_{(-\sum_{i=2}^m \frac{a_{i}}{a_1}\bar{f}_{i})^c}\right)}{\mu_1}.
    \end{equation}
\eqref{eq:stationary-by-pf-po} can be used to construct the $\dot{\mathbb{H}}^1$-gradient of $\D$, which can be used in the framework introduced in \cite{Jacobs2020BF}. An application on the classical Wasserstein barycenter problem can be found in \cite{kim2025sobolev}, which scales w.r.t the number of marginals better than the multimarginal dual formulation used in \cite{Zhou2022Efficient}.

As a comparison, we say $\bar{f}_{[I\setminus \{1\}]}$ is a \textit{saddle point} to $\D$ if
    \[\sup_{f_{[I_+\setminus \{1\}]}}\inf_{f_{[I_-]}}\D(f)=\inf_{f_{[I_-]}}\sup_{f_{[I_+\setminus \{1\}]}}\D(f)=\D(\bar{f}_{[I_+\setminus \{1\}]},\bar{f}_{[I_-]}).\]
Any saddle point is a stationary point by definition.  

We consider the linearized functional of $\D$ at $\bar{f}_{[I]}$, defined by
    \begin{equation}\label{eq:linear-dual-func-po-f}
    \overline{\D}(f_{[I\setminus \{1\}]}; \bar{f}_{[I]})=\sum_{i\in I\setminus \{1\}} \int_{\X} a_i f_i^c\, \mathrm{d}\mu_i + \int_{\X} a_1  \bar{f}_{1}^c\, \mathrm{d}\mu_1 + a_1 \int_{\Y} (\bar{f}_1 - f_1)\, \mathrm{d}\bar{\nu}.
    \end{equation}
where $\bar{f}_1, f_1$ are redundant potentials as in \eqref{eq:redundant} and $\bar{\nu}=\push{\left(T_{\bar{f}^c_1}\right)}{\mu_1}$. It holds that $\overline{\D}(\bar{f}_{[I\setminus \{1\}]};\bar{f}_{[I]})=\D(\bar{f}_{[I\setminus \{1\}]})$. Moreover, $\overline{\D}(f;\bar{f})$ has the convex-concave property, thanks to Lemma \ref{lem:OT-dual-optimality}. This facilitates us to characterize the stationary points to $\D$ via the characterization of the saddle points to $\overline{\D}$.

\begin{proposition}\label{prop:saddle}

$\bar{f}_{[I\setminus \{1\}]}$ is a saddle point to $\overline{\D}(f_{[I\setminus \{1\}]};\bar{f}_{[I]})$ if and only if for every $i\in I\setminus \{1\}$, there exists $\pi_{i}\in\Pi(\mu_{i},\bar{\nu})$  such that 
    \begin{equation}\label{eq:saddle-optimality}
        \pi_{i}(x_{i},y)(\bar{f}^c_{i}(x_{i})+ \bar{f}_{i}(y)-c(x_{i},y))=0.
    \end{equation}
Furthermore, for every $i\in I$
    \begin{equation}\label{eq:saddle-strong-duality}
    K_c(\mu_{i},\bar{\nu})=\int_{\X\times\Y} c(x_{i},y)\, \mathrm{d}\pi_{i}(x_{i},y)=\int_{\X} \bar{f}^c_{i}\, \mathrm{d}\mu_{i}+\int_{\Y} \bar{f}_{i}\, \mathrm{d}\bar{\nu}=D_c(\mu_i,\bar{\nu}).
    \end{equation}
    
\end{proposition}
\begin{proof}

We plug the definition of redundant potential $f_1$ into \eqref{eq:linear-dual-func-po-f}
    \[
    \overline{\D}(f_{[I\setminus \{1\}]};\bar{f}_{[I]})=\sum_{i\in I\setminus \{1\}}\left[\int_{\X} a_i f_i^c\, \mathrm{d}\mu_i + \int_{\Y}a_i f_i\, \mathrm{d}\bar{\nu}\right]+\textup{constant},
    \]
and note that this linearized functional is separable w.r.t its inputs $f_{[I\setminus \{1\}]}$. Thus maximizing $\overline{\D}(f_{[I\setminus \{1\}]};\bar{f})$ w.r.t maximizing variables and minimizing $\overline{\D}(f_{[I\setminus \{1\}]};\bar{f}_{[I]})$ w.r.t minimizing variables can be carried out in a coordinate-wise manner. Consequently, the equivalence holds by the strong duality from Theorem \ref{thm:Brenier}.

\end{proof}

\begin{proposition}\label{prop:stationary}

$\bar{f}_{[I\setminus \{1\}]}$ is a stationary point to $\D(f)$ if and only if $\bar{f}_{[I\setminus \{1\}]}$ is a saddle point to $\overline{\D}(f_{[I\setminus \{1\}]};\bar{f}_{[I]})$.

\end{proposition}
\begin{proof}
Since the linearized functional $\overline{\D}(f;\bar{f})$ has the convex-concave property, we just need to show that $\bar{f}_{[I\setminus \{1\}]}$ is a stationary point to $\D(f)$ if and only if $\bar{f}_{[I\setminus \{1\}]}$ is a stationary point to $\overline{\D}(f;\bar{f})$. Recall $J_{\mu}^c (f)=\int_{\X} f^c\, \mathrm{d}\mu$. 

We first examine the variation on the maximizing variables. For any tuple of bounded continuous functions $f_{[I_+ \setminus \{1\}]}$, we have
    \[
    \begin{aligned}
    &\D(\bar{f}_{[I_+ \setminus \{1\}]}+\varepsilon f_{[I_+ \setminus \{1\}]},\bar{f}_{[I_-]})-\D(\bar{f}_{[I_+ \setminus \{1\}]},\bar{f}_{[I_-]})=\mathsf{C}(\varepsilon)+\Rm{1}(\varepsilon),\\
    &\overline{\D}(\bar{f}_{[I_+ \setminus \{1\}]}+\varepsilon f_{[I_+ \setminus \{1\}]},\bar{f}_{[I_-]}; \bar{f}_{[I]}) -\overline{\D}(\bar{f}_{[I_+ \setminus \{1\}]},\bar{f}_{[I_-]};\bar{f}_{[I]})=\mathsf{C}(\varepsilon)+\overline{\Rm{1}}(\varepsilon),
    \end{aligned}
    \]
where we denote the common term $\sum_{i \in I_+\setminus \{1\}} a_i [J^c_{\mu_i}(\bar{f}_i+\varepsilon f_i)-J^c_{\mu_i}(\bar{f}_i)]$ by $\mathsf{C}(\varepsilon)$ and denote the different terms $a_1 [J^c_{\mu_1}(\bar{f}_1 - \varepsilon\sum_{i\in I_+\setminus \{1\}} \frac{a_i}{a_1}f_i)-J^c_{\mu_1}(\bar{f}_1)]$ and $a_1 \varepsilon\int_{\Y} ( \sum_{i\in I_+\setminus \{1\}} \frac{a_i}{a_1}f_i)\,\mathrm{d}\bar{\nu}$ by $\Rm{1}(\varepsilon)$ and $\overline{\Rm{1}}(\varepsilon)$ respectively. 

By Lemma \ref{lem:OT-dual-optimality}, the first variation of the functional $J_{\mu_1}^c(f)$ at $\bar{f}_1=-\sum_{i=2}^{m}\frac{a_{i}}{a_1}\bar{f}_{i}$ is given by $\delta J^c_{\bar{f}_1;\mu_1}=-\push{\left(T_{\bar{f}^c_1}\right)}{\mu_1}=-\bar{\nu}$. Thus
    \begin{equation*}
    J^c_{\mu_1}(\bar{f}_1 - \varepsilon\sum_{i\in I_+\setminus \{1\}} \frac{a_i}{a_1}f_i)-J^c_{\mu_1}(\bar{f}_1)=-\int_{\Y} (-\varepsilon \sum_{i\in I_+\setminus \{1\}} \frac{a_i}{a_1}f_i)\, \mathrm{d}\bar{\nu}+o(\varepsilon)
    \end{equation*}
yields that  $\Rm{1}(\varepsilon) =\overline{\Rm{1}}(\varepsilon) + o(\varepsilon)$. Thus, 
    \[
    \begin{aligned}
    &\lim_{\varepsilon\to 0}\frac{\D(\bar{f}_{[I_+ \setminus \{1\}]}+\varepsilon f_{[I_+ \setminus \{1\}]},\bar{f}_{[I_-]})-\D(\bar{f}_{[I_+ \setminus \{1\}]},\bar{f}_{[I_-]})}{\varepsilon}\\
    =&\lim_{\varepsilon\to 0} \frac{\mathsf{C}(\varepsilon)+\Rm{1}(\varepsilon)}{\varepsilon}
    =\lim_{\varepsilon \to 0} \frac{\mathsf{C}(\varepsilon)+\overline{\Rm{1}}(\varepsilon)}{\varepsilon} + \lim_{\varepsilon\to 0}\frac{o(\varepsilon)}{\varepsilon}\\
    =&\lim_{\varepsilon\to 0}\frac{\overline{\D}(\bar{f}_{[I_+ \setminus \{1\}]}+\varepsilon f_{[I_+ \setminus \{1\}]},\bar{f}_{[I_-]}; \bar{f}_{[I]}) -\overline{\D}(\bar{f}_{[I_+ \setminus \{1\}]},\bar{f}_{[I_-]};\bar{f}_{[I]})}{\varepsilon},
    \end{aligned}
    \]
which proves the equivalence for maximizing variables in the definition of stationary points. For the variation on the minimizing variables, the proof is analogous.
\end{proof}

Combine Proposition \ref{prop:saddle} and \ref{prop:stationary}, the stationary point $\bar{f}_{[I\setminus \{1\}]}$ to $\D$ satisfies \eqref{eq:saddle-optimality} and \eqref{eq:saddle-strong-duality}. Consequently, $\D(\bar{f})=\B(\bar{\nu})$. However, it is still not sufficient to show the strong duality \eqref{eq:duality}, since we do not yet know whether $\inf\sup \D(f)=\D(\bar{f})$ and $\inf \B(\nu)=\B(\bar{\nu})$ holds or not.

\subsection{Properties of saddle points}\label{subsec:global}

In this section, we will show how to construct the signed barycenter from the saddle point of the dual problem, thereby ensuring the strong duality $\inf \B(\nu)=\sup\inf \D(f)$. We also propose sufficient conditions for the uniqueness of the signed barycenter.  

\begin{theorem}[Any saddle point induces an optimal signed barycenter]\label{thm:main-induced} 

Suppose $\mu_1\in \Pc^{\textup{AC}}(\X)$ and $\bar{f}_{[I\setminus \{1\}]}$ is a $c$-concave saddle point to $\D(f)$. Then 
$\bar{\nu} = \push{T_{(-\sum_{i=2}^m \frac{a_{i}}{a_1}\bar{f}_{i})^c}}{\mu_1}$ defined in \eqref{eq:bary-po-f} is an optimal signed barycenter to \eqref{eq:bary-func}, and the strong duality \eqref{eq:duality} holds. Furthermore, for any marginal $\mu_i\in \Pc^{\textrm{AC}}(\X)$, $\bar{\nu}=\push{T_{\bar{f}_{i}^c}}{\mu_{i}}$.

\end{theorem}

\begin{proof}
Apply Lemma \ref{lem:dual},
    \[
    \begin{aligned}
    \mathscr{B}(\bar{\nu})&\geqslant \inf_{\nu}\mathscr{B}(\nu)=\inf_{\nu}\inf_{f_{[I_-]}}\sup_{f_{[I_+ \setminus \{1\}]}} a_1 K_c(\mu_1,\nu) + \sum_{i=2}^m a_{i} \int_{\X} f_{i}^c \,\mathrm{d}\mu_i + \int_{\Y}(\sum_{i=2}^m a_{i}f_{i})\,\mathrm{d}\nu\\
    &\geqslant \sup_{f_{[I_+ \setminus \{1\}]}}\inf_{f_{[I_-]}}\D(f)=\D(\bar{f}_{[I_+\setminus \{1\}]},\bar{f}_{[I_-]})=\sum_{i=2}^m \int_{\X} a_{i} \bar{f}^c_{i}\, \mathrm{d}\mu_{i} +  \int_{\X} a_1 \bar{f}_1^c \, \mathrm{d}\mu_1.
    \end{aligned}
    \]

The saddle point $\bar{f}_{[I\setminus \{1\}]}$ is also a stationary point, thanks to the characterization of stationary points to $\D(f)$ in Subsection \ref{subsec:stationary}, from \eqref{eq:saddle-strong-duality} we have
    \[
    \begin{aligned}
    \int_{\X\times \Y} c(x_{i},y)\, \mathrm{d}\pi_{i}(x_{i},y)&=\int_{\X} \bar{f}^c_{i}\, \mathrm{d}\mu_{i}+\int_{\Y} \bar{f}_{i}\, \mathrm{d}\bar{\nu};\\
    \int_{\X\times \Y} c(x_{1},y)\, \mathrm{d}\pi_{1}(x_{1},y)&=\int_{\X} \bar{f}^c_{1}\, \mathrm{d}\mu_{1}+\int_{\Y} \bar{f}_{1}\, \mathrm{d}\bar{\nu},
    \end{aligned}\]
where $\pi_{i}\in \Pi(\mu_i,\bar{\nu})$ is optimal. Substituting these identities into the above inequality gives
    \[
    \begin{aligned}
        \mathscr{B}(\bar{\nu})&\geqslant \sum_{i=2}^m \int_{\X} a_{i} \bar{f}^c_{i}\, \mathrm{d}\mu_{i} + \int_{\X} a_1 \bar{f}_1^c \, \mathrm{d}\mu_1 + a_1 [\int_{\X\times \Y} c(x_1,y)\, \mathrm{d}\pi_1(x_1,y) - \int_{\X} \bar{f}^c_1 \mathrm{d}\mu_1 - \int_{\Y} \bar{f}_1 \mathrm{d}\bar{\nu} ]\\
        &=\int_{\X\times \Y} a_1 c(x_1,y)\, \mathrm{d}\pi_1(x_1,y) + \sum_{i=2}^m \int_{\X} a_{i} \bar{f}^c_{i}\, \mathrm{d}\mu_{i} -a_1 \int_{\Y} (-\sum_{i=2}^m \frac{a_{i}}{a_1}\bar{f}_{i})\, \mathrm{d}\bar{\nu}\\
        &=\int_{\X\times \Y} a_1 c(x_1,y)\, \mathrm{d}\pi_1(x_1,y)+\sum_{i=2}^m \int_{\X\times \Y} a_{i} c(x_{i},y)\, \mathrm{d}\pi_{i}(x_{i},y)\\
        &= \sum_{i=1}^m a_{i}K_c(\mu_{i},\bar{\nu})=\mathscr{B}(\bar{\nu}),
        \end{aligned}
        \]
which completes the strong duality, i.e., $\inf \B(\nu)=\sup\inf \D(f)$.
\end{proof}

Given a saddle point $\bar{f}_{[I\setminus \{1\}]}$, the following theorem states that for  $i\in I_+\setminus \{1\}$, the OT plan between \textit{any} signed barycenter and $\mu_i$ belongs to the $c$-superdifferential of $\bar{f}_i$. This is reminiscent of a related property in the classical theory of optimal transport (see Remark 1.15 in \cite{Ambrosio2013User}). We emphasize that this property holds for potentials associated with a \textit{positive weight}, in contrast to Proposition \ref{prop:saddle} where the induced signed barycenter $\bar{\nu}$ applies. With an additional absolutely continuous marginal associated with a \textit{positive weight}, we obtain the uniqueness of the signed barycenter from the perspective of dual approach, complementing the uniqueness result from Theorem \ref{thm:convex-unique}.

\begin{theorem}[Uniqueness of the signed barycenter]\label{thm:unique}
Suppose $\mu_1\in \Pc^{\textup{AC}}(\X)$ and $\bar{f}_{[I\setminus \{1\}]}$ is a $c$-concave saddle point to $\D(f)$. Any signed barycenter $\hat{\nu}$ (including $\bar{\nu}$ induced from $\bar{f}_{[I\setminus \{1\}]}$) to $\B$ necessarily satisfies that for $i\in I_+\setminus \{1\}$, there exists $\pi_i\in \Pi(\mu_i,\hat{\nu})$ such that $\pi_i(x_i,y)(c(x_i,y)-\bar{f}_i^c(x_i)-\bar{f}_i(y))=0$.

If there is $\mu_i\in \Pc^{\textup{AC}}(\X)$ for some $i\in I_+\setminus \{1\}$, then the signed barycenter is unique.
\end{theorem}

\begin{proof}
Let $\bar{\nu}=\push{T_{(-\sum_{i=2}^m \frac{a_i}{a_1}\bar{f}_i)^c}}{\mu_1}$ be the signed barycenter induced from the saddle point $\bar{f}_{[I\setminus \{1\}]}$ and $\hat{\nu}$ be another signed barycenter such that $\B(\hat{\nu})=\B(\bar{\nu})$.  

For $i\in I\setminus \{1\}$, let $\hat{f}_{i}$ be the Kantorovich potential such that $K_c(\mu_i, \hat{\nu})=\int_{\X} \hat{f}_i^c\, \mathrm{d}\mu_i + \int_{\Y}\hat{f}_i\, \mathrm{d}\hat{\nu}=\sup_{f}\int_{\X} f^c\, \mathrm{d}\mu_i + \int_{\Y}f \, \mathrm{d}\hat{\nu}$. Then
    \begin{align*}
    \B(\hat{\nu})&=\sum_{i=2}^m a_i [\int_{\X} \hat{f}_i^c\, \mathrm{d}\mu_i + \int_{\Y}\hat{f}_i \,\mathrm{d}\hat{\nu}] + a_1 K_c(\mu_1,\hat{\nu})\\
    &\geqslant \sum_{i\in I_+\setminus \{1\}} a_i [\int_{\X}\bar{f}_i^c\, \mathrm{d}\mu_i + \int_{\Y}\bar{f}_i\, \mathrm{d}\hat{\nu}] + \sum_{i\in I_-} a_i [\int_{\X}\hat{f}_i^c\, \mathrm{d}\mu_i + \int_{\Y}\hat{f}_i\, \mathrm{d}\hat{\nu}]+ a_1 K_c(\mu_1,\hat{\nu})\\
    &\geqslant \inf_{\nu} \sum_{i\in I_+\setminus \{1\}} a_i [\int_{\X}\bar{f}_i^c\, \mathrm{d}\mu_i + \int_{\Y}\bar{f}_i\, \mathrm{d}\nu] + \sum_{i\in I_-} a_i [\int_{\X}\hat{f}_i^c \,\mathrm{d}\mu_i + \int_{\Y}\hat{f}_i\, \mathrm{d}\nu]+ a_1 K_c(\mu_1,\nu)\\
    &= \D(\bar{f}_{[I_+\setminus \{1\}]}, \hat{f}_{[I_-]})\geqslant \inf_{f_{[I_-]}}\D(\bar{f}_{[I_+\setminus \{1\}]}, f_{[I_-]})=\D(\bar{f}_{[I_+\setminus\{1\}]}, \bar{f}_{[I_-]})=\B(\bar{\nu}),
    \end{align*}
which implies that all above inequalities are equalities. 

As a result, for $i\in I_+\setminus\{1\}$, $\bar{f}_{i}$ is the Kantorovich potential to $K_c(\mu_i,\hat{\nu})$. Thanks to the optimality condition in Theorem \ref{thm:Brenier}, there exists $\pi_i\in \Pi(\mu_i,\hat{\nu})$ such that $
\pi_i(x_i,y)(c(x_i,y)-\bar{f}^c_i(x_i)-\bar{f}_i(y))=0$.

If $\mu_i\in \Pc^{\textup{AC}}(\X)$ for some $i\in I_+\setminus \{1\}$, then the above property can be rephrased in terms of OT maps. Namely, $T_{\bar{f}^c_i}$ is the optimal transport map from $\mu_i$ to $\hat{\nu}$. On the other hand, thanks to Theorem \ref{thm:main-induced}, $T_{\bar{f}^c_i}$ is the optimal transport map from $\mu_i$ to $\bar{\nu}$. Thus $\hat{\nu}=\bar{\nu}$.

\end{proof}

\subsection{A criterion for upgrading stationary points to saddle points}\label{subsec:upgrade}

In practice, it is often more tractable to obtain a stationary point from the minimax problem. For instance, via \eqref{eq:stationary-by-pf-po} when all input measures are absolutely continuous. We are therefore interested in testing whether this stationary point is a saddle point.

The objective value of the dual problem \eqref{eq:dual-func-po-f} is the same whether $f_{[I\setminus \{1\}]}$ are $c$-concave or merely bounded continuous, since replacing each $f_i$ by $f_i^{cc}$ increases the objective value when $f_i$ is a maximizing variable and decreases the objective value when $f_i$ is a minimizing variable. Thus we may consider another essentially equivalent dual functional
    \[
    D(f_{[I_+\setminus \{1\}]}, g_{[I_-]})=\sum_{i=2}^{\ell}\int_{\X}a_i f_i^c\,\mathrm{d}\mu_i + \sum_{i=\ell+1}^{m}\int_{\X} a_i g_i\,\mathrm{d}\mu_i + a_1 \int_{\X} (-\sum_{i=2}^{\ell}\frac{a_i}{a_1}f_i + \sum_{i=\ell+1}^{m}\frac{\abs{a_i}}{a_1} g_i^c )^c \,\mathrm{d}\mu_1 ,
    \]
which is essentially the composition of \eqref{eq:dual-func-po-f} with the change of variables $g_i=f_i^c$ for $i\in I_-$. Crucially,
    \[\sup_{f_{[I_+\setminus \{1\}]}}\inf_{f_{[I_-]}}\D(f_{[I_+\setminus \{1\}]},f_{[I_-]})=\sup_{f_{[I_+\setminus \{1\}]}} \inf_{g_{[I_-]}}D(f_{[I_+\setminus \{1\}]},g_{[I_-]}).\]

Thanks to the concavity of $\D$ with respect to the maximizing variables, to show that a given $c$-concave stationary point $\bar{f}_{[I\setminus \{1\}]}$ is a saddle point,  it suffices to check that
\[\D(\bar{f}_{[I_+\setminus \{1\}]}, \bar{f}_{[I_-]})=\inf_{f_{[I_-]}} \D(\bar{f}_{[I_+\setminus \{1\}]}, f_{[I_-]}).\]
Defining $\bar{g}_i=\bar{f}_i^c$, this is equivalent to checking that
\[
D(\bar{f}_{[I_+\setminus \{1\}]},\bar{g}_{[I_-]}) =\inf_{g_{[I_-]}}D(\bar{f}_{[I_+\setminus \{1\}]},g_{[I_-]}).
\]
The upshot of working with $D$ and the $g$ variables is that in Proposition \ref{prop:cond_convex} we will be able to show that under certain conditions on $\bar{f}_{[I_+\setminus \{1\}]}$ and the cost $c$,
\[
g_{[I_-]}\mapsto D(\bar{f}_{[I_+\setminus \{1\}]}, g_{[I_-]})
\]
is convex along curves in $(C_b(\X))^{m-\ell}$.  Then, in Theorem \ref{thm:saddle_pt_sufficient} we will complete the argument that the infimum is achieved at $\bar{g}_{[I_-]}$ by transferring the stationarity properties of the $\bar{f}_{[I\setminus \{1\}]}$ to show that the $\bar{g}_{[I_-]}$ are a critical point of $D$. 
\vspace{.1in}

\begin{proposition}\label{prop:cond_convex}
Let $\mu_1\in \Pc^{\textup{AC}}(\X)$ and $\bar{f}_{[I\setminus \{1\}]}$ be a $c$-concave stationary point of $\D$.  Define
  \begin{equation}\label{eq:htest}
h(y;x_1, \mathbf{x}^-):=a_1c(x_1,y)-\sum_{i\in I_-} |a_i|c(x_i,y)+\sum_{i\in I_+\setminus\{1\}} a_i \bar{f}_i(y).
    \end{equation}
Suppose that for every fixed $(x_1, \mathbf{x}^-)\in \X\times\mathcal{X}^{|I_-|}$,
the mapping $y\mapsto h(y;x_1, \mathbf{x}^-)$ is convex and bounded from below. Then $D(\bar{f}_{[I_+\setminus \{1\}]}, g_{[I_-]})$ is convex along curves in the set of bounded functions, i.e., for any bounded continuous functions $g_{[I_-]}$, define $g_{i,t}:=t g_i+(1-t)\bar{f}_i^c$ for $i\in I_-$, then
\[
t\mapsto D(\bar{f}_2,\ldots, \bar{f}_{\ell}, g_{\ell+1,t},\ldots, g_{m,t})
\]
is convex.
\end{proposition}
\begin{remark}
The functional $\D$ is in general not necessarily jointly convex w.r.t the minimizing variables. However, at a stationary point, the functional $\D$ is convex w.r.t the minimizing variables along some curves, if the weighted cost and maximizing potentials satisfy the convexity criterion. 
\end{remark}

\begin{proof}
Note that $\int_{\X} a_i g_i \,\mathrm{d}\mu_i$ is linear for $i\in I_-$, thus our goal is equivalent to showing that
    \[
    t\mapsto \int_{\X} a_1\Big(-\sum_{i\in I_+\setminus \{1\}} \frac{a_i}{a_1}\bar{f_i}+\sum_{i\in I_-} \frac{\abs{a_i}}{a_1}g_{i,t}^c\Big)^c(x_1) \,\mathrm{d} \mu_1(x_1)
    \]
is convex.
    
Expanding the integrand in terms of the $c$-transforms, we see that
    \[
    \begin{aligned}
    & \int_{\X} a_1\Big(-\sum_{i\in I_+\setminus \{1\}} \frac{a_i}{a_1}\bar{f_i} + \sum_{i\in I_-} \frac{\abs{a_i}}{a_1}g_{i,t}^c\Big)^c(x_1) \,\mathrm{d} \mu_1(x_1)\\
    =&\inf_{Y} \int_{\X} \Big( a_1c(x_1, Y(x_1))+\sum_{i\in I_+ \setminus \{1\}} a_i \bar{f}_i(Y(x_1))- \sum_{i\in I_-} \abs{a_i} g_{i,t}^c(Y(x_1)) \Big)\,\mathrm{d} \mu_1(x_1)\\
    =&\inf_{Y}\int_{\X} \Big( a_1c(x_1, Y(x_1))+\sum_{i\in I_+ \setminus \{1\}} a_i \bar{f}_i(Y(x_1))+ \sum_{i\in I_-}\sup_{x_i} \abs{a_i}\big(g_{i,t}(x_i)-c(x_i,Y(x_1))  \Big)\,\mathrm{d} \mu_1\\
 = &\inf_{Y} \sup_{\gamma\in \Pi(\mu_1,\cdot) }\int_{\X\times \X^{|I_-|}} \Big( h(Y(x_1); x_1,\mathbf{x}^-) + \sum_{i\in I_-} \abs{a_i} g_{i,t}(x_i)  \Big)\,\mathrm{d} \gamma(x_1,\mathbf{x}^-),
    \end{aligned}
    \]
where $Y:\X\to\Y$ and $\gamma$ is a joint coupling over $\X\times\X^{|I_-|}$, whose first marginal is $\mu_1$. If we choose some compact, convex set $K\subset \Y$ and restrict $Y$ to take values in $K$, then the last line is bounded from above by 
\[
\inf_{Y:\X\to K} \sup_{\gamma\in \Pi(\mu_1,\cdot) }\int_{\X\times \X^{|I_-|}} \Big( h(Y(x_1); x_1,\mathbf{x}^-) + \sum_{i\in I_-} \abs{a_i} g_{i,t}(x_i)  \Big)\,\mathrm{d} \gamma(x_1,\mathbf{x}^-).
\]
Clearly, the space of maps $Y:\X\to K$ is a weakly compact subset of  $L^1(\mu_1)$.  Since $h$ is convex with respect to the $y$ variable, it follows that the functional is weakly lower semicontinuous with respect to $Y$.  It is also clear that the functional is concave with respect to $\gamma$ and weakly upper semicontinuous with respect to the weak convergence of $\gamma$. 

Thus, by Sion's minimax theorem, we can interchange the order of inf and sup. Defining
    \[
H_K(x_1,\mathbf{x}^-):=\inf_{y\in K} h(y; x_1, \mathbf{x}^{-}),
    \]
we have 
\[
D(\bar{f}_2,\ldots, \bar{f}_{\ell}, g_{\ell+1,t},\ldots, g_{m,t})\leqslant \sup_{\gamma\in \Pi(\mu_1,\cdot) }\int_{\X\times \X^{|I_-|}} \Big( H_K(x_1,\mathbf{x}^-)+ \sum_{i\in I_-} \abs{a_i} g_{i,t}(x_i)  \Big)\,\mathrm{d} \gamma(x_1,\mathbf{x}^-).
\]
We then have the following string of inequalities
    \begin{align*}
       &\sup_{\gamma\in \Pi(\mu_1,\cdot) }\int_{\X\times \X^{|I_-|}} \Big( H_K(x_1,\mathbf{x}^-) +  \sum_{i\in I_-} \abs{a_i} g_{i,t}(x_i)  \Big)\,\mathrm{d} \gamma(x_1,\mathbf{x}^-)\\
       \leqslant &\sup_{\gamma_1, \gamma_0\in \Pi(\mu_1,\cdot) }t\int_{\X\times \X^{|I_-|}} \Big( H_K(x_1,\mathbf{x}^-) + \sum_{i\in I_-} \abs{a_i} g_i(x_i)  \Big)\,\mathrm{d} \gamma_1(x_1,\mathbf{x}^-)\\
       &\hspace{3cm}+(1-t)\int_{\X\times \X^{|I_-|}} \Big( H_K(x_1,\mathbf{x}^-) + \sum_{i\in I_-} \abs{a_i}\bar{f}_i^c(x_i)  \Big)\,\mathrm{d} \gamma_0(x_1,\mathbf{x}^-)\\
       \leqslant& \inf_{\substack{Y_1:\X\to K\\ Y_0:\X\to K} }\sup_{\gamma_1, \gamma_0\in \Pi(\mu_1,\cdot) } t\int_{\X\times \X^{|I_-|}} \Big( h(Y_1(x_1); x_1,\mathbf{x}^-) + \sum_{i\in I_-} \abs{a_i} g_{i}(x_i)  \Big)\,\mathrm{d} \gamma_1(x_1,\mathbf{x}^-)\\
       &\hspace{3cm}+(1-t)\int_{\X\times \X^{|I_-|}} \Big( h(Y_0(x_1); x_1,\mathbf{x}^- ) + \sum_{i\in I_-} \abs{a_i} \bar{f}_i^c(x_i)  \Big)\,\mathrm{d} \gamma_0(x_1,\mathbf{x}^-)\\
       =&\inf_{\substack{Y_1:\X\to K\\ Y_0:\X\to K}} t\int_{\X} \Big( a_1c(x_1, Y_1 (x_1))+\sum_{i\in I_+ \setminus \{1\}} a_i \bar{f}_i(Y_1(x_1))- \sum_{i\in I_-} \abs{a_i} g_{i}^c(Y_1 (x_1)) \Big)\,\mathrm{d} \mu_1(x_1)\\
       &\hspace{0.7cm}+(1-t)\int_{\X} \Big( a_1c(x_1, Y_0(x_1))+\sum_{i\in I_+ \setminus \{1\}} a_i \bar{f}_i(Y_0(x_1))- \sum_{i\in I_-} \abs{a_i} \bar{f}_{i}^c(Y_0(x_1)) \Big)\,\mathrm{d} \mu_1(x_1)
    \end{align*}
Letting $K$ approach $\Y$, we have now shown that
\[
t\mapsto D(\bar{f}_2,\ldots, \bar{f}_{\ell}, g_{\ell+1,t},\ldots, g_{m,t})
\]
is convex.

\end{proof}

In the following theorem, we further assume that the signed barycenter is absolutely continuous. A sufficient condition for this assumption is provided in Proposition \ref{prop:AC-hold}. 

\begin{theorem}[Sufficient condition for a saddle point]\label{thm:saddle_pt_sufficient}
Let $\mu_1\in \Pc^{\textup{AC}}(\X)$ and $\bar{f}_{[I\setminus \{1\}]}$ be a $c$-concave stationary point of $\D$. For \[
h(y;x_1, \mathbf{x}^-):=a_1c(x_1,y)-\sum_{i\in I_-} |a_i|c(x_i,y)+\sum_{i\in I_+\setminus\{1\}} a_i \bar{f}_i(y)
\]
defined in \eqref{eq:htest}, suppose that for every fixed $(x_1, \mathbf{x}^-)\in \X\times\mathcal{X}^{|I_-|}$,
the mapping $y\mapsto h(y;x_1, \mathbf{x}^-)$ is convex and bounded from below. 
 Assume that $\bar{\nu} = \push{T_{(-\sum_{i=2}^m \frac{a_{i}}{a_1}\bar{f}_{i})^c}}{\mu_1}$ defined in \eqref{eq:bary-po-f} is absolutely continuous, then $\bar{f}_{[I\setminus \{1\}]}$ is a saddle point of $\D$.
\end{theorem}
\begin{proof}
Let $H(x_1,\mathbf{x}^-):=\inf_{y\in \Y} h(y; x_1, \mathbf{x}^-)$ which is bounded from below by our assumption.  For any continuous function $\psi:\X^{|I_-|}\to\R$, we introduce a new $H$-based transform 
    \[
    \psi^H(x_1):=\sup_{\mathbf{x}^-\in \X^{|I_-|}} H(x_1,\mathbf{x}^-)-\psi(\mathbf{x}^-),
    \]
and for any bounded continuous function $g_i$ for $i\in I_-$, it follows that 
    \[
    \int_{\X} a_1\Big(-\sum_{i\in I_+\setminus \{1\}} \frac{a_i}{a_1}\bar{f_i}-\sum_{i\in I_-} \frac{a_i}{a_1}g_i^c\Big)^c(x_1) \, \mathrm{d} \mu_1(x_1)\geqslant \int_{\X} \Big(\sum_{i\in I_-} a_i g_i\Big)^H(x_1) \, \mathrm{d} \mu_1(x_1).
    \]
Hence, 
     \begin{equation*}
     \begin{aligned}
         &D(\bar{f}_2,\ldots, \bar{f}_{\ell},  g_{\ell+1},\ldots, g_m)\\
         \geqslant &\sum_{i\in I_+\setminus\{1\}} a_i\int_{\X} \bar{f}_i^c(x_i)\, \mathrm{d}\mu_i(x_i)+ \sum_{i\in I_-} a_i\int_{\X} g_i(x_i)\, \mathrm{d}\mu_i(x_i)+ \int_{\X} \Big(\sum_{i\in I_-} a_i g_i\Big)^H(x_1) \,\mathrm{d} \mu_1(x_1)\\
         \geqslant &\int_{\X^m} \Big(\sum_{i\in I_+\setminus\{1\}} a_i\bar{f}_i^c(x_i) + H(x_1, \mathbf{x}^-)\Big) \, \mathrm{d}\pi(x_1, x_2, \ldots, x_m),
     \end{aligned}
     \end{equation*}
for any $\pi\in \Pi(\mu_1,\mu_2,\ldots,\mu_m)$. Clearly,
    \[
    \D(\bar{f}_{[I_+\setminus \{1\}]},\bar{f}_{\ell+1},\ldots, \bar{f}_m)= D(\bar{f}_{[I_+\setminus \{1\}]}, \bar{g}_{\ell+1},\ldots,\bar{g}_{m}).
    \] 
Thus, if we can find a coupling $\bar{\pi}\in \Pi(\mu_1,\ldots,\mu_m)$ such that 
\begin{equation*}
\D(\bar{f}_{[I\setminus \{1\}]})\leqslant \int_{\X^m} \Big(\sum_{i\in I_+\setminus\{1\}} a_i\bar{f}_i^c(x_i) + H(x_1, \mathbf{x}^-)\Big) \,\mathrm{d}\bar{\pi}(x_1, x_2, \ldots, x_m),
\end{equation*}
then we will have shown that $\bar{f}_{[I\setminus \{1\}]}$ is a saddle point.

Using the characterization of a stationary point in Subsection \ref{subsec:stationary}, each $\bar{f}_i$ is an optimal Kantorovich potential for the transport of $\bar{\nu}$ to $\mu_i$ for all $i\in \{1,\ldots, m\}$ satisfying \eqref{eq:saddle-strong-duality}.
By the gluing lemma, there exists a coupling $\gamma\in \Pi_*(\mu_1,\ldots,\mu_m; \bar{\nu})$ such that for all $i\in \{1,\ldots, m\}$
    \[
    \int_{\X^m\times\Y} \Big( c(x_i,y)-\bar{f}_i(y)-\bar{f}_i^c(x_i)\Big)\, \mathrm{d}\gamma(x_1,\ldots, x_m,y)=0.
    \]
Furthermore, by Proposition \ref{prop:gg_derivative}, we may disintegrate 
    \[
    \mathrm{d}\gamma(x_1,\ldots, x_m,y)=\mathrm{d} \pi^+_{(\mathbf{x}_+,y)}(\mathbf{x}_-) \mathrm{d}\nu^+(\mathbf{x}_+,y)
    \]
and for $\nu^+$-almost every $(\mathbf{x}_+,y)\in \X^{|I_+|}\times\Y$, we have  
    \[
    \int_{\X^{|I_-|}}\sum_{i=1}^m a_i\nabla_y c(x_i,y)\, \mathrm{d}\pi^+_{(\mathbf{x}^+,y)}(\mathbf{x}_-)=0.
    \]
Let us choose $\bar{\pi}=P_{\#}\gamma$ where $P:\X^m\times\Y\to\X^m$ is the projection map that forgets $\Y$. From our assumptions on $h$,  given some $\varepsilon>0$ and any $(x_1,\mathbf{x}_-)\in \X\times\X^{|I_-|}$, there exists a point $y_{\varepsilon}(x_1,\mathbf{x}_-)\in \Y$ such that
    \[
    H(x_1,\mathbf{x}_-)+\varepsilon >h(y_{\varepsilon}(x_1,\mathbf{x}_-); x_1,\mathbf{x}_- ).
    \]
Therefore,
    \begin{align*}
    &\varepsilon+\int_{\X^m} \Big(\sum_{i\in I_+\setminus\{1\}} a_i\bar{f}_i^c(x_i) + H(x_1, \mathbf{x}^-)\Big) \,\mathrm{d}\bar{\pi}(x_1, \ldots, x_m)\\
    >&\int_{\X^m} \Big(\sum_{i\in I_+\setminus\{1\}} a_i\bar{f}_i^c(x_i) + h(y_{\varepsilon}(x_1,\mathbf{x}_-) ; x_1,\mathbf{x}_-)\Big) \,\mathrm{d}\bar{\pi}(x_1, \ldots, x_m)\\
    =&\int_{\X^m\times \Y} \Big(\sum_{i\in I_+\setminus\{1\}} a_i\big(c(x_i,y)-\bar{f}_i(y)\big) + h(y_{\varepsilon}(x_1,\mathbf{x}_-) ; x_1,\mathbf{x}_-)\Big) \,\mathrm{d}\gamma(x_1, \ldots, x_m,y).
    \end{align*}
Note that $\sum_{i\in I_+\setminus\{1\}} a_i\big(c(x_i,y)-\bar{f}_i(y)\big)=\sum_{i=1}^m a_ic(x_i,y)-h(y; x_1,\mathbf{x}_-)$.  
Since 
    \[
    0\in \partial_y \Big(\sum_{i\in I_+\setminus\{1\}} a_i\big(c(x_i,y)-\bar{f}_i(y)\big)\Big)
    \]
for $\gamma$-almost every $(x_1,\ldots, x_m,y)$, we must have 
    \[
    0\in \partial_y \Big(\sum_{i=1}^m a_ic(x_i,y)-h(y; x_1,\mathbf{x}_- )\Big)
    \]
for $\gamma$-almost every $(x_1,\ldots, x_m,y)$. Combining this with the convexity of $y\mapsto h(y; x_1,\mathbf{x}_-)$, it follows that
    \[
    \begin{aligned}
    h(y_{\varepsilon}(x_1,\mathbf{x}_-); x_1,\mathbf{x}_- )&\geqslant h(y; x_1,\mathbf{x}_-) +  \sum_{i=1}^m a_i\nabla_y c(x_i,y)\cdot ( y_{\varepsilon}(x_1,\mathbf{x}_-)-y)\\
    &=h(y; x_1,\mathbf{x}_-) + L_{(x_1,\mathbf{x}_-,y)}(y_{\varepsilon}(x_1,\mathbf{x}_-)-y),
    \end{aligned}
    \]
where we denote the linear map by $L_{(x_1,\mathbf{x}_-,y)}:\R^d\to\R$ for notational simplicity. 
Now we use this inequality to continue the previous long inequality in the above
    \begin{align*}
    &\int_{\X^m\times \Y} \Big(\sum_{i\in I_+\setminus\{1\}} a_i\big(c(x_i,y)-\bar{f}_i(y)\big) + h(y_{\varepsilon}(x_1,\mathbf{x}_-) ; x_1,\mathbf{x}_-)\Big) \,\mathrm{d}\gamma(x_1, \ldots, x_m,y)\\
    \geqslant & \int_{\X^m\times \Y} \Big(\sum_{i\in I_+\setminus\{1\}} a_i\big(c(x_i,y)-\bar{f}_i(y)\big) + h(y; x_1,\mathbf{x}_-)+L_{(x_1,\mathbf{x}_-,y)}(y_{\varepsilon}(x_1,\mathbf{x}_-)-y)\Big) \,\mathrm{d}\gamma\\
    =& \int_{\X^m\times \Y} \Big(\sum_{i=1}^m a_ic(x_i,y)+L_{(x_1,\mathbf{x}_-,y)}(y_{\varepsilon}(x_1,\mathbf{x}_-)-y)\Big) \,\mathrm{d}\gamma(x_1, \ldots, x_m,y).
    \end{align*}
Since $\bar{f}_i$ is differentiable almost everywhere and $\bar{\nu}$ is absolutely continuous, it follows that $\nabla_y c(x_i,y)-\nabla \bar{f}_i(y)=0$ for $\gamma$-almost every $(x_1,\ldots, x_m,y)\in \X^m\times \Y$.
In particular, this gives us
    \[
    a_1\nabla_y c(x_1,y)-\sum_{i\in I_-} |a_i|\nabla_y c(x_i,y)+\sum_{i\in I_+\setminus\{1\}} a_i \nabla \bar{f}_i(y)=a_1\nabla \bar{f}_1(y)+\sum_{i=2}^m a_i \nabla \bar{f}_i(y)=0
    \]
for $\gamma$-almost every $(x_1,\ldots, x_m,y)\in \X^m\times \Y$.  Thanks to the convexity of $y\mapsto a_1 c(x_1,y)-\sum_{i\in I_-} |a_i| c(x_i,y)+\sum_{i\in I_+\setminus\{1\}} a_i \bar{f}_i(y)$, this critical point condition implies that
    \[
    H(x_1,\ldots, x_m)=a_1 c(x_1,y)-\sum_{i\in I_-} |a_i| c(x_i,y)+\sum_{i\in I_+\setminus\{1\}} a_i \bar{f}_i(y)
    \]
for $\gamma$-almost every $(x_1,\ldots, x_m,y)\in \X^m\times \Y$. Here we slightly abuse the notation $H(x_1,\ldots,x_m)=H(x_1,\mathbf{x}^-)$. Using this equality, we see that
    \[
    \begin{aligned}
    &\int_{\X^m} \Big(\sum_{i\in I_+\setminus\{1\}} a_i\bar{f}_i^c(x_i) + H(x_1,\ldots, x_m)\Big) \, \mathrm{d}\bar{\pi}(x_1, \ldots, x_m)\\
    =&\int_{\X^m\times\Y} \Big(a_1 c(x_1,y)-\sum_{i\in I_-} |a_i| c(x_i,y)+\sum_{i\in I_+\setminus\{1\}} a_i \big(\bar{f}_i(y)+\bar{f}_i^c(x_i)\big)\Big)\, \mathrm{d}\gamma(x_1,\ldots, x_m,y)\\
    =&\int_{\X^m\times\Y} \Big(a_1 c(x_1,y)+\sum_{i=2}^m a_i \bar{f}_i(y)\Big)\, \mathrm{d}\gamma(x_1,\ldots, x_m,y)+\sum_{i=2}^m a_i\int_{\X} \bar{f}_i^c(x_i)\, \mathrm{d}\mu_i(x_i)\\
    \geqslant & a_1\int_{\X} \Big(-\frac{1}{a_1}\sum_{i=2}^m a_i \bar{f}_i\Big)^c\, \mathrm{d} \mu_1(x_1)+\sum_{i=2}^m a_i\int_{\X} \bar{f}_i^c(x_i)\, \mathrm{d}\mu_i(x_i)\\
    =&\D(\bar{f}_2,\ldots, \bar{f}_{\ell}, \bar{f}_{\ell+1},\ldots, \bar{f}_m),
    \end{aligned}
    \]
    where the penultimate inequality is simply from the definition of the $c$-transform. Hence, we have shown that 
    \[
    \D(\bar{f}_2,\ldots, \bar{f}_{\ell}, \bar{f}_{\ell+1},\ldots, \bar{f}_m)\leqslant \int_{\X^m} \Big(\sum_{i\in I_+\setminus\{1\}} a_i\bar{f}_i^c(x_i) + H(x_1,\ldots, x_m)\Big) \,\mathrm{d}\pi(x_1, \ldots, x_m)
    \]
which completes the argument that $\bar{f}_{[I\setminus \{1\}]}$ is a saddle point.  The optimality of $\bar{\nu}$ follows from Theorem \ref{thm:main-induced}.

\end{proof}

\begin{remark}
For the quadratic cost $c(x,y)=\frac{1}{2}\abs{x-y}^2$, given a $c$-concave stationary point $\bar{f}_{[I\setminus \{1\}]}$, to check whether $h(y;x_1,\mathbf{x}^-)$ is convex, it is equivalent to test if maximizing variables $\bar{f}_{[I_+\setminus \{1\}]}$ associated with positive weights satisfy
    \begin{equation}\label{eq:additional-po}
    y\mapsto \frac{\abs{y}^2}{2}-\sum_{i\in I_+\setminus \{1\}} a_i(\frac{\abs{y}^2}{2}-\bar{f}_i(y)) \mathrm{~is~convex}.
    \end{equation}
This is done by direct computation from the characterization of $c$-concave function for the quadratic cost. Furthermore, if $\bar{f}_{[I_+\setminus \{1\}]}$ are second-order differentiable, the criterion \eqref{eq:additional-po} can be simplified to
    \begin{equation}\label{eq:hessian_criteria}
    \sum_{i\in I_+\setminus \{1\}} a_i D^2 \bar{f}_i\succeq (-a_1+\sum_{i \in I_-} \abs{a_i})I=(-1+\sum_{i\in I_+\setminus\{1\}}a_i)I.
    \end{equation}

In addition, the redundant potential $\bar{f}_1=-\sum_{i=2}^m \frac{a_{i}}{a_{1}}\bar{f}_{i}$ is $c$-concave, since 
    \begin{align*}
    &\frac{\abs{y}^2}{2}-\bar{f_1}(y)=\frac{\abs{y}^2}{2}+\sum_{i=2}^m \frac{a_{i}}{a_1}\bar{f}_{i}(y)=\frac{\abs{y}^2}{2}+\sum_{i\in I_+\setminus\{1\}} \frac{a_{i}}{a_1}\bar{f}_{i}(y)- \sum_{i\in I_-}\frac{\abs{a_{i}}}{a_1}\bar{f}_i(y)\\ 
    &=\frac{1-\sum_{i=2}^m a_i}{a_1}\frac{\abs{y}^2}{2}+\sum_{i\in I_+\setminus\{1\}} \frac{a_{i}}{a_1}\bar{f}_{i}(y)- \sum_{i\in I_-}\frac{\abs{a_{i}}}{a_1}\bar{f}_{i}(y)\\
    &=\frac{1}{a_1}\left[\frac{\abs{y}^2}{2}- \sum_{i\in I_+\setminus \{1\}} a_i (\frac{\abs{y}^2}{2}-\bar{f}_i(y))+ \sum_{i\in I_-}\abs{a_i} (\frac{\abs{y}^2}{2}-\bar{f}_i(y))\right]
    \end{align*}
which is a sum of convex functions, and thus is $c$-concave by the characterization again.  Note that if the ordering relation in \eqref{eq:hessian_criteria} is strict, then it also follows that $h$ is bounded from below in $y$. 
\end{remark}

\begin{proposition}\label{prop:AC-hold}
Under our assumptions on the sets $\X,\Y$ and the cost $c$, let $f$ be the optimal $c$-concave Kantorovich potential and $\pi$ the optimal plan for the transport of $\nu\in \Pc(\Y)$ to $\mu\in \Pc(\X)$.  Suppose that $\mu$ is absolutely continuous and $x\mapsto \nabla_y c(x,y)$ has a uniformly Lipschitz inverse.    If there exists a countable collection of compact sets $K_n\subset \X\times \Y$ and constants $C_n>0$ such that  $\pi(\bigcup_{n\in \mathbb{N}} K_n)=1$ and
\begin{equation}\label{eq:semi_bound}
\frac{\abs{f(y+v)+f(y-v)-2f(y)}}{|v|^2}\leqslant C_n
\end{equation}
for all vectors $v$ with sufficiently small norm and every $y\in P_2(K_n)$, where $P_2$ denote the projection onto the second coordinate. Then
$\nu$ must be absolutely continuous with respect to the Lebesgue measure.
\end{proposition}
\begin{proof}

Recall that the $c$-superdifferential $\partial^c f(y)$ of some $c$-concave function $f$ is defined by 
\[\partial^c f(y)=\left\{x\in \X: y\in\argmin_{z\in \Y} c(x,z)-f(z)\right\}.\]
That is, $x\in \partial^c f(y)$ implies that $c(x,y)-f(y)\leqslant c(x,z)-f(z)$ for any $z\in \Y$.

Choose some $y\in P_2(K_n)$ and note that if $x_1,x_0\in \partial^c f(y)$, then 
\[
\begin{aligned}
f(y+v)+f(y-v)-2f(y)&\leqslant c(x_1,y+v)-c(x_1,y)+c(x_0,y-v)-c(x_0,y)\\  
&=\big(\nabla_y c(x_1,y)-\nabla_y c(x_0,y)\big)\cdot v+O(|v|^2)
\end{aligned}
\]
If $x_1\neq x_0$, then the twist condition implies that $\nabla_y c(x_1,y)\neq \nabla_y c(x_0,y)$.  Hence, after choosing $v$ appropriately, there exists some $\delta>0$ independent of $|v|$ such that
\[
f(y+v)+f(y-v)-2f(y)\leqslant -\delta|v| +O(|v|^2)
\]
contradicting our assumption on $f$.  Therefore, $\partial^c  f(y)$ is a singleton for all $y\in P_2(K_n)$.  
Since $\pi(\bigcup_{n\in \mathbb{N}} K_n)=1$, it follows that there exists a map $S:\Y\to\X$ such that $\pi=(S(y),y)_{\#}\nu$.  Furthermore, if we define $G_n:=\{y\in \Y: (S(y),y)\in K_n\}$, then it follows that $\nu(\bigcup_{n\in \mathbb{N}} G_n)=1$.

Choose some $y\in G_n$.  Define $\vp_y(v):=f(y+v)-f(y)-\nabla_y c(S(y),y)\cdot v$. The $c$-concavity of $f$ implies that 
\[
\vp_y(v)\leqslant c(S(y),y+v)-c(S(y),y)-\nabla_y c(S(y),y)\cdot v.
\]
Thus, by the compactness of $K_n$ and $c\in C^2(\X\times \Y)$, there exists some $B_n>0$ such that 
\[
\vp_y(v)\leqslant B_n|v|^2
\]
for all $(S(y),y)\in K_n$. On the other hand, the bound \eqref{eq:semi_bound} implies that 
\[
|\vp_y(v)+\vp_y(-v)|\leqslant C_n|v|^2
\]
for all $v$ sufficiently small.  Therefore, combining this with the $c$-concavity bound, we see that 
\begin{equation}\label{eq:derivative_bound}
|\vp_y(v)|\leqslant (B_n+C_n)|v|^2
\end{equation}
for all $v$ sufficiently small.  Thus, $f$ is differentiable on $G_n$ and we have the equality
\[
\nabla_y c(S(y),y)=\nabla f(y).
\]

From \eqref{eq:derivative_bound}, it follows that for any $y\in G_n$ and any $v$ sufficiently small that
\[
|f(y+v)-f(y)-\nabla f(y)\cdot v|\leqslant (C_n+B_n)|v|^2 
\]
Hence, if we take two points $y,y'\in G_n$ with $|y-y'|$ sufficiently small, then the previous inequality can be used to deduce that 
\[
|f(y')-f(y)+\nabla f(y')\cdot (y-y')+(\nabla f(y')-\nabla f(y))\cdot v|\leqslant (C_n+B_n)(|v|^2+|v+y-y'|^2)
\]
Using the inequality one last time on $f(y')-f(y)+\nabla f(y')\cdot (y-y')$, it follows that 
\[
|(\nabla f(y')-\nabla f(y))\cdot v|\leqslant (C_n+B_n)(|v|^2+|v+y-y'|^2+|y-y'|^2).
\]
Now if we choose $v=|y-y'|u$ where $u$ is a unit vector parallel to $\nabla f(y')-\nabla f(y)$, we find that
\[
|\nabla f(y)-\nabla f(y')|\leqslant 6(C_n+B_n)|y-y'|,
\]
allowing us to conclude that $\nabla f$ is uniformly Lipschitz on $G_n$.  From our assumption on the cost, this implies that $S$ is uniformly Lipschitz on $G_n$.   Since Lipschitz maps preserve Lebesgue sets of measure zero, if $N\subset \Y$ is a Lebesgue null set, then 
\[
\nu(N)\leqslant \sum_{n\in \mathbb{N}} \nu(N\cap G_n)\leqslant \sum_{n\in \mathbb{N}} \mu\Big(S(N\cap G_n)\Big)=0.
\]
\end{proof}

From Theorem \ref{thm:saddle_pt_sufficient} and Proposition \ref{prop:AC-hold}, we obtain a sufficient condition for a unique signed barycenter induced from the stationary point of the dual problem, which can be very practical for numerical computation of a  signed barycenter.

\begin{corollary}\label{cor:sufficient}

Let $\mu_1,\mu_2 \in \Pc^{\textup{AC}}(\X)$. Under the assumptions on $\bar{f}_{[I\setminus \{1\}]}$ and $h(y;x_1,\mathbf{x}^-)$ defined in \eqref{eq:htest} in Theorem \ref{thm:saddle_pt_sufficient}. If $x\mapsto \nabla_y c(x,y)$ has a uniformly Lipschitz inverse 
then $\bar{\nu} = \push{\left(T_{\bar{f}^c_{2}}\right)}{\mu_2}$ defined in \eqref{eq:bary-po-f} induced from $\bar{f}_{[I\setminus \{1\}]}$ is the unique optimal signed barycenter, which is absolutely continuous with respect to the Lebesgue measure.

\end{corollary}
\begin{proof}

If we can use Proposition \ref{prop:AC-hold} on  $\bar{f}_2$ to show that $\bar{\nu}$ is absolutely continuous, then combining Theorem \ref{thm:unique} and Theorem \ref{thm:saddle_pt_sufficient}, it would follow that $\bar{\nu}=\push{\left(T_{\bar{f}^c_{2}}\right)}{\mu_2}$ induced from the stationary point $\bar{f}_{[I\setminus \{1\}]}$ is the unique signed barycenter.

To apply Proposition \ref{prop:AC-hold}, let $E_n\subset \X$ and $F_n\subset \Y$ be sequences of compact sets such that $\bigcup_{n\in \mathbb{N}} E_n=\X$ and $\bigcup_{n\in \mathbb{N}} F_n=\Y$.  We then define 
\[
K_n:=E_n\times\Big( F_n\cap \{y\in \Y: \partial^c \bar{f}_2(y)\cap E_n\neq \varnothing \; \}\Big),
\]
where we note that $\{y\in \Y: \partial^c \bar{f}_2(y)\cap E_n\neq \varnothing \; \}$ is closed since $\bar{f}_2$ is upper semicontinuous.
It then follows that given any optimal plans $\pi_2$ we have $\pi_2(\bigcup_{n\in \mathbb{N}} K_n)=1$.  From the $c$-concavity of $\bar{f}_2$ we have
\[
\bar{f}_2(y+v)+\bar{f}_2(y-v)-2\bar{f}_2(y)\leqslant c(x_2(y),y+v)+c(x_2(y),y-v)-2c(x_2(y),y)
\]
for any $x_2(y)\in \partial^c f_2(y)$.  Thanks to the compactness of $K_n$, there must exist $B_n>0$ such that 
\[
\bar{f}_2(y+v)+\bar{f}_2(y-v)-2\bar{f}_2(y)\leqslant B_n|v|^2
\]
for all $y\in P_2(K_n)$ and $|v|$ sufficiently small.

On the other hand, from the convexity of $h$ at any $x_1, \mathbf{x}^-$ and the semiconcavity of the $\bar{f}_2$ on $P_2(K_n)$, it follows that $\bar{f}_2$ must satisfy \eqref{eq:semi_bound} on $P_2(K_n)$. Thus, we are done.

\end{proof}

\backmatter

\bmhead{Acknowledgements}

We would like to thank Nicol\'{a}s Garc\'{i}a Trillos and Katy Craig for helpful discussions. B.Z would also like to thank Ricardo Baptista and Giulia Cavagnari for their suggestions. 

\bmhead{Funding}

M.J is partially supported by NSF grant DMS-2400641.

\section*{Declarations}

\bmhead{Conflict of interest}

The authors have no relevant financial or non-financial interests to disclose.

\bibliography{reference.bib}

\end{document}